\documentclass{article}
\usepackage{times,amsmath,amsbsy,amsthm, amssymb}
\usepackage{graphicx,color,float, epsfig}

\newtheorem{theorem}{Theorem}[section]
\newtheorem{lemma}[theorem]{Lemma}

\newtheorem{remark}[theorem]{Remark}

\numberwithin{equation}{section}

\begin{document}

\title{Robust recovery-type {\em a posteriori} error estimators for streamline upwind/Petrov Galerkin discretizations for singularly perturbed problems}

\author{
    Shaohong Du\thanks{
        School of Mathematics and Statistics, Chongqing Jiaotong University, Chongqing 400074, China.
        E-mail: {dushaohong@csrc.ac.cn}, {duzheyan.student@sina.com.cn}.
        Research partially supported by the National Natural Science Foundation of China under grants 91430216, 11471031.
    }
    \and
    Runchang Lin\thanks{
        Department of Mathematics and Physics, Texas A\&M International University, Laredo, Texas 78041, USA.
        E-mail: {rlin@tamiu.edu}.
        Research partially supported by the US National Science Foundation grant 1217268 and by a University Research Grant of Texas A\&M International University.
    }
    \and
    Zhimin Zhang\thanks{
        Beijing Computational Science Research Center, Beijing 100193, China and Department of Mathematics, Wayne State University, Detroit, Michigan 48202, USA. E-mail: {zmzhang@csrc.ac.cn}, {zzhang@math.wayne.edu}.
        Research partially supported by the US National Science Foundation through grant DMS-1419040 and by the National Natural Science Foundation of China under grants 91430216, 11471031.
    }
}
\date{}

\maketitle

\begin{abstract}
    In this paper, we investigate adaptive streamline upwind/Petrov Galerkin (SUPG) methods for singularly perturbed convection-diffusion-reaction equations in a new dual norm presented in \cite{Du}. The flux is recovered by either local averaging in conforming $H({\rm div})$ spaces or weighted global $L^{2}$ projection onto conforming $H({\rm div})$ spaces. We further introduce a recovery stabilization procedure, and develop completely robust {\em a posteriori} error estimators with respect to the singular perturbation parameter $\varepsilon$. Numerical experiments are reported to support the theoretical results and to show that the estimated errors depend on the degrees of freedom uniformly in $\varepsilon$.
\end{abstract}



\section{Introduction}
Let $\Omega$ be a bounded polygonal or polyhedral domain in ${\mathbb{R}}^{d}$ ($d=2$ or $3$) with
Lipschitz boundary $\Gamma=\Gamma_{D}\cup\Gamma_{N}$, where $\Gamma_{D} \cap \Gamma_{N} = \emptyset$. Consider
the following stationary singularly perturbed convection-diffusion-reaction problem
\begin{equation}\label{PDE1}
    \left \{
    \begin{array}{rl}
        \mathcal{L}u := -\varepsilon\triangle u+{\bf a}\cdot\nabla u+bu=f & \mbox{in}~\Omega,\\
        u=0 & \mbox{on}~\Gamma_{D},\\
        \displaystyle \varepsilon \frac{\partial u}{\partial{\bf n}} = g & \mbox{on}~\Gamma_{N},
    \end{array}\right.
\end{equation}
where $0<\varepsilon\ll 1$ is the singular perturbation
parameter, ${\bf a}\in (W^{1,\infty}(\Omega))^d$, $b\in L^{\infty}(\Omega)$, $f\in L^2(\Omega)$, ${\bf n}$ is the outward
unit normal vector to $\Gamma$, and equation \eqref{PDE1} is scaled such
that $||{\bf a}||_{L^\infty} = \mathcal{O}(1)$ and $||b||_{L^\infty} = \mathcal{O}(1)$. The Dirichlet
boundary $\Gamma_{D}$ has a positive $(d-1)$-dimensional Lebesgue measure, which includes the inflow
boundary $\{{\bf x}\in \partial \Omega : {\bf a}({\bf x}) \cdot {\bf n}<0\}$. Assume that there are
two nonnegative constants $\beta$ and $c_{b}$, independent of $\varepsilon$, satisfying
\begin{equation}\label{Assumption}
    b-\frac{1}{2}\nabla\cdot{\bf a} \geq \beta \quad \text{and} \quad ||b||_{L^{\infty}{(\Omega)}}\leq c_{b}\beta.
\end{equation}
Note that if $\beta=0$, then $b\equiv0$ and there is no reaction term in \eqref{PDE1}.

Adaptive finite element methods (FEMs) for numerical solutions of partial differential equations (PDEs) are very
popular in scientific and engineering computations. {\em A posteriori} error estimation is an essential
ingredient of adaptivity. Error estimators in literature can be categorized into three classes: residual
based, gradient recovery based, and hierarchical bases based. Each approach has certain advantages.

Designing a robust {\em a posteriori} error estimator for singularly
perturbed equations is challenging, because the estimators usually
depend on the small diffusion parameter $\varepsilon$. This problem
was first investigated by Verf\"{u}rth \cite{Verfurth2}, in which
both upper and lower bounds for error estimator in an
$\varepsilon$-weighted energy norm was proposed. It was  shown that
the estimator was robust when the local P\'{e}clet number is not
very large. Generalization of this approach can be found in, e.g.,
\cite{Berrone02,Kunert,Rapin,San01}. He considered also robust
estimators in an ad hoc norm in \cite{Verfurth1}. In \cite{San08},
Sangalli pointed out that the ad hoc norm may not be appropriate for
problem (\ref{PDE1}), and proposed a residual-type {\em a
posteriori} estimator for 1D convection-diffusion problem which is
robust up to a logarithmic factor with respect to global P\'{e}clet
number. Recently, John and Novo \cite{John-Novo} proposed a robust
{\em a posteriori} error estimator in the natural SUPG norm (used in
the {\em a priori} analysis) under some hypotheses, which, however,
may not be fulfilled in practise. In \cite{Ani}, a fully computable,
guaranteed upper bounds are developed for the discretisation error
in energy norm. Very recently, Tobiska and Verf\"{u}rth
\cite{Tobiska} presented robust residual {\em a priori} error
estimates for a wide range of stabilized FEMs.

For {\em a posteriori} error estimation of singularly perturbed
problems, it is crucial to employ an appropriate norm, since the
efficiency of a robust estimator depends fully on the norm. Du and Zhang \cite{Du} proposed a
dual norm, which is induced by an $\varepsilon$-weighted energy norm
and a related $H^{1/2}(\Omega)$-norm. A uniformly robust {\em a
posteriori} estimator for the numerical error was obtained from the
new norm. Both theoretical and numerical results showed that
the estimator performs better than the existing ones in the
literature.

It is well known that the {\em a posteriori} error estimators of the recovery type possess many appealing properties,
including simplicity, university, and asymptotical exactness, which lead to their widespread adoption, especially
in the engineering community (cf., e.g., \cite{Ainsworth,Ainsworth0,Bank2,Carstensen,Zhang1,Zhang,Zienkiewicz,Zienkiewicz2}). However,
when applied to many problems of practical interest, such as interface singularities, discontinuities in the form of
shock-like fronts, and of interior or boundary layers, they lose not only asymptotical exactness but also efficiency
on relatively coarse meshes. They may overrefine regions where there are no error, and hence fail to reduce the
global error (see \cite{Bank,Ovall1,Ovall2}). To overcome this difficulty, Cai and Zhang \cite{cai1} developed a
global recovery approach for the interface problem. The flux is recovered in $H({\rm div})$ conforming finite
element (FE) spaces, such as the Raviart-Thomas (RT) or the Brezzi-Douglas-Marini (BDM) spaces, by global weighted $L^{2}$-projection or
local averaging. The resulting recovery-based (implicit and explicit) estimators are measured in the standard energy
norm, which turned out to be robust if the diffusion coefficient is monotonically distributed.

This approach was further extended for solving general second-order elliptic PDEs \cite{cai2}. The implicit
estimators based on the $L^{2}$-projection and $H({\rm div})$ recovery procedures were proposed to be the sum
of the error in the standard energy norm and the error of the recovered flux in a weighted $H({\rm div})$ norm.
The global reliability and the local efficiency bounds for these estimators were established. For singularly
perturbed problems, the estimators developed in \cite{cai1,cai2} are not robust with respect to $\varepsilon$. To the authors' knowledge, no robust recovery-type estimators have been proposed for such problems in the literature.

Motivated by aforementioned works, we extend the approach in \cite{Du} and develop robust recovery-based {\em a posteriori} error estimators for the SUPG method for singularly perturbed problems. Three procedures will be applied, which are the explicit recovery through local averaging in $RT_0$ spaces, the implicit recovery based on the global weighted $L^{2}$-projection in $RT_0$ and $BDM_1$ spaces, and the implicit $H({\rm div})$ recovery procedure. Numerical errors will be measured in a dual norm presented in \cite{Du}. Note that these estimators are different from those in \cite{Du}, since the jump in the normal component of the flux consists of a recovery indicator in addition to an incidental term (see Remark 4.1). Our recovery procedures are also different from those in \cite{cai1,cai2} (e.g., the flux recovery based on the local averaging provides an appropriate choice of weight factor, the $H({\rm div})$ recovery procedure develops a stabilization technique, the recovery procedures treat Neumann boundary conditions properly, etc.). Moreover, the estimators developed here are uniformly robust with respect to $\varepsilon$ and $\beta$.

The rest of this paper is organized as follows. In Section \ref{Sec_VarForm}, we introduce the variational
formulation and some preliminary results. In Section \ref{FluxRec}, we define an implicit flux recovery
procedure based on the $L^{2}$-projection onto the lowest-order $RT$ or $BDM$ spaces, and an explicit
recovery procedure through local averaging in the lowest-order $RT$ spaces. In Section \ref{APostErr},
for implicit and explicit recovery procedures, we give a reliable upper bound for the numerical error
in a dual norm developed in \cite{Du}. Section \ref{Sec_analysis} is devoted to the analysis of efficiency
of the estimators. Here, the efficiency is in the sense that the converse estimate of upper bound holds
up to different higher order terms (usually oscillations of data) and a different multiplicative constant
depends only on the shape of the mesh. We show that the estimators are completely robust with respect
to $\varepsilon$ and $\beta$. In Section \ref{HdivRec}, we define a stabilization $H({\rm div})$ recover
procedure, and develop robust recovery-based estimator by using the main results of Sections~\ref{APostErr}
and \ref{Sec_analysis}. Numerical tests are provided in Section \ref{NumExp} to support the theoretical results.

\section{Variational Formulation and Preliminary Results}\label{Sec_VarForm}
For any subdomain $\omega$ of $\Omega$ with a Lipschitz boundary $\gamma$, denote by $<\cdot, \cdot>_e$
and $(\cdot,\cdot)_{\omega}$ the inner products on $e\subseteq\gamma$ and $\omega$, respectively. Throughout
this paper, standard notations for Lebesgue and Sobolev spaces and their norms and seminorms are
used \cite{ADAMS}. In particular, for $1\leq p <\infty$ and $0<s<1$, the norm of the fractional
Sobolev space $W^{s,p}(\omega)$ is defined as
\begin{equation*}
    ||v||_{W^{s,p}(\omega)} := \Big\{ ||v||_{L^{p}(\omega)}^{p}+
    \int_{\omega}\int_{\omega}\frac{|v(x)-v(y)|^{p}}{|x-y|^{d+ps}}dxdy \Big\}^{1/p}\ \ {\rm for}\ \ v\in W^{s,p}(\omega).
\end{equation*}
When $p=2$, we write $H^{s}(\omega)$ for $W^{s,2}(\omega)$. We will also use the space
$
    H({\rm div};\omega): =\{\tau\in L^{2}(\omega)^{d}: \nabla\cdot \tau\in L^{2}(\omega)\}.
$
To simplify notations, we write $||\cdot||_{s,\omega} = ||\cdot||_{H^s(\omega)}$, $|\cdot|_{s,\omega} = |\cdot|_{H^s(\omega)}$,
and $||\cdot||_{\gamma} = ||\cdot||_{L^{2}(\gamma)}$. Moreover, when no confusion may arise, we will omit the
subindex $\Omega$ in the norm and inner product notations if $\omega=\Omega$. Let $H_{D}^{1}(\Omega) :=\{v\in H^{1}(\Omega) : v|_{\Gamma_{D}}=0 \}$. Define a bilinear form $B(\cdot,\cdot)$ on $H_{D}^{1}(\Omega)\times H_{D}^{1}(\Omega)$ by
\begin{equation}\label{PDE2}
    B(u,v)=\varepsilon(\nabla u,\nabla v)+({\bf a}\cdot\nabla u,v)+(bu,v).
\end{equation}
The variational formulation of (\ref{PDE1}) is to find $u\in H_{D}^{1}(\Omega)$ such that
\begin{equation}\label{PDE3}
    B(u,v)=(f,v) + <g,v>_{\Gamma_{N}} \quad \forall v\in H_{D}^{1}(\Omega).
\end{equation}
Under the assumption \eqref{Assumption}, equation (\ref{PDE3}) possesses a unique weak solution (cf., e.g., \cite{Roos}).

Let $\mathcal{T}_{h}$ be a shape regular admissible triangulation of $\Omega$ into triangles or tetrahedra satisfying the angle condition \cite{Ciarlet}. We use $F\preceq G$ to represent $F \leq CG$ , and write $F\thickapprox G$ if both $F\preceq G$ and $G\preceq F$ hold true. Here and in what follows, we use $C$ for a generic positive constant depending only on element shape regularity and $d$. Assume that $\mathcal{T}_{h}$ aligns with the partition of $\Gamma_{D}$ and $\Gamma_{N}$. Let $\mathcal{E}$ be the set of all {\em edges} (for $d=2$) or {\em faces} (for $d=3$) of elements in $\mathcal{T}_{h}$. Then $\mathcal{E}=\mathcal{E}_{\Omega}\cup\mathcal{E}_{D}\cup\mathcal{E}_{N}$, where $\mathcal{E}_{\Omega}$ is the set of interior edges/faces, and $\mathcal{E}_{D}$ and $\mathcal{E}_{N}$ are the sets of boundary edges/faces on $\Gamma_{D}$ and $\Gamma_{N}$, respectively. Let $P_{k}(K)$ be the space of polynomials on $K$ of total degree at most $k$. Let the FE space $V_{h}$ be
\begin{equation*}
    V_{h} :=\{v_{h}\in C(\overline{\Omega}) : v_{h}|_{K}\in P_{1}(K) \quad \forall K\in\mathcal{T}_{h}, \, v_{h}|_{\Gamma_{D}}=0\}.
\end{equation*}

Define a bilinear form $B_{\delta}(\cdot,\cdot)$ on $V_{h}\times V_{h}$ and a linear functional $l_{\delta}(\cdot)$ on $V_{h}$ by
\begin{align*}
    & B_{\delta}(u_{h},v_{h}) = B(u_{h},v_{h}) + \sum_{K\in\mathcal{T}_{h}}\delta_{K} (-\varepsilon\triangle u_{h}
    +{\bf a}\cdot\nabla u_{h}+bu_{h}, {\bf a}\cdot\nabla v_{h})_{K}, \\
    & l_{\delta}(v_{h})=(f,v_{h}) + <g,v_{h}>_{\Gamma_{N}} + \sum_{K\in\mathcal{T}_{h}} \delta_{K}(f,{\bf a}\cdot\nabla v_{h})_{K},
\end{align*}
where $\delta_{K}$'s are nonnegative stabilization parameters satisfying
\begin{equation}\label{PDE5}
    \delta_{K}||{\bf a}||_{L^{\infty}(K)}\leq Ch_{K} \quad \forall K\in\mathcal{T}_{h}.
\end{equation}
Note that $\triangle u_{h}$ is interpreted as the Laplacian applied to $u_{h}|_{K}, \forall K\in\mathcal{T}_{h}$.
For the lowest-order element, though $\triangle u_{h}$ vanishes on each element, we will keep this term for
complete presentation of the SUPG method and its analysis in below (cf. Section~\ref{Sec_analysis}).

Then the FE approximation of (\ref{PDE1}) is to find $u_{h}\in V_{h}$ such that
\begin{equation}\label{PDE6}
    B_{\delta}(u_{h},v_{h})=l_{\delta}(v_{h}) \quad \forall v_{h}\in V_{h}.
\end{equation}
Note that the choice $\delta_{K}=0$ for all $K\in\mathcal{T}_{h}$ yields the standard Galerkin method, and
the choice $\delta_{K}>0$ for all $K$ corresponds to the SUPG-discretization. The existence and uniqueness of
solution to (\ref{PDE6}) are guaranteed by \eqref{Assumption} and \eqref{PDE5} (cf., e.g., \cite{Franca, Hughes,Verfurth1}).

Define an $\varepsilon$-weighted energy norm by
\begin{equation*}
||v||_{\varepsilon} :=(\varepsilon|v|_{1}^{2}+\beta||v||^{2})^{1/2} \quad \forall v\in H^{1}(\Omega).
\end{equation*}
Let $h({\bf x})$ be a function satisfying $0<h_{\rm min}\leq h({\bf x})\leq h_{\rm max}<\infty$ almost
everywhere in $\Omega$. Define a norm of $v\in H_{D}^{1}(\Omega)$ with respect to $h({\bf x})$ by
\begin{align*}
        & |||v|||^2:=||v||_{\varepsilon}^{2}+\max \big\{ ||v||_{H^{1/2}(\Omega)}^{2},\; ||h({\bf x})^{-1/2}v||^{2} + ||h({\bf x})^{1/2}\nabla v||^{2} \big\} \\
        \text{or} \hspace{0.5in}
        & |||v|||^{2} :=||v||_{\varepsilon}^{2}+||h({\bf x})^{-1/2}v||^{2}+||h({\bf x})^{1/2}\nabla v||^{2}.
\end{align*}
It is shown \cite{Du} that the dual norm
\begin{equation}\label{dual norm}
    |||\cdot|||_{*} :=\sup_{v\in H_{D}^{1}(\Omega)\setminus\{0\}}\frac{B(\cdot,v)}{|||v|||}
\end{equation}
induced by the bilinear form (\ref{PDE2}) satisfies, for $u\in H_{D}^{1}(\Omega)$,
\begin{equation*}
    |||u|||_{*}=\sup_{v\in H_{D}^{1}(\Omega)\setminus\{0\}}\frac{<\mathcal{L}u,v>}{|||v|||}
\geq\frac{|||u|||}{||\mathcal{L}^{-1}||_{(H_{D}^{1}(\Omega))^{*}\rightarrow
H_{D}^{1}(\Omega)}}.
\end{equation*}
This inequality shows that $|||u|||_{*}$ may reflect the first derivatives of $u$ even if $\varepsilon=0$.

Let $I_{h} :L^{2}(\Omega)\rightarrow V_{h}$ be the Cl\'{e}ment interpolation operator (cf. \cite{Clement,Verfurth1,San08}
and \cite[Exercise 3.2.3]{Ciarlet}). The following estimates on $I_{h}$ are found in \cite{Du}.
\begin{lemma}\label{convection 1}
    Let $h_{e}$ be the diameter of an edge/face $e$. For any $v\in H_{D}^{1}(\Omega)$,
    \begin{align}
        \label{convection 2}
        & \sum_{K\in\mathcal{T}_{h}}\delta_{K}^{2}\max\{\beta,\varepsilon h_{K}^{-2},h_{K}^{-1}\} ||{\bf a}\cdot\nabla(I_{h}v)||_{K}^{2} \preceq |||v|||^{2}, \\
        \label{convection 3}
        & \sum_{K\in\mathcal{T}_{h}}\max\{\beta,\varepsilon h_{K}^{-2},h_{K}^{-1}\}||v-I_{h}v||_{K}^{2}\preceq|||v|||^{2}, \\
        \label{convection 4}
        & \sum_{e\subset\Gamma_{N}}\max\{\varepsilon^{1/2}\beta^{1/2},\varepsilon h_{e}^{-1},1\}||v-I_{h}v||_{e}^{2}\preceq|||v|||^{2}.
    \end{align}
\end{lemma}

\begin{remark}[On the norm $|||\cdot|||_{*}$]\label{Rk:Norms}
We first review a robust residual-based {\em a posteriori} estimator, which is proposed in SUPG norm under some
hypotheses \cite{John-Novo}. Let
\begin{equation*}
\begin{array}{c}
    \eta_{1}=\displaystyle\Big(\sum\limits_{K\in\mathcal{T}_{h}}\min\big\{\frac{C}{\beta},C\frac{h_{K}^{2}}{\varepsilon},
    24\delta_{K}\big\}||R_{K}||_{K}^{2}\Big)^{1/2},\quad
    \eta_{2}=\displaystyle\Big(\sum\limits_{K\in\mathcal{T}_{h}}24\delta_{K}||R_{K}||_{K}^{2}\Big)^{1/2},\\
    \text{and~~}
    \eta_{3}=\displaystyle\Big(\sum\limits_{e\in\mathcal{E}}\min\big\{\frac{24}{||{\bf a}||_{\infty,e}},
    C\frac{h_{e}}{\varepsilon}, \frac{C}{\varepsilon^{1/2}\beta^{1/2}} \big\}||R_{e}||_{e}^{2}\Big)^{1/2},
\end{array}
\end{equation*}
where the cell residual $R_K$ and edge/face residual $R_e$ are defined by \eqref{Eq:CellResidl} and
\begin{equation*}
 R_{e} :=\left \{ \begin{array}{ll}
  -[\varepsilon\nabla u_{h}\cdot{\bf n}_{e}]|_{e}\ \ \  & \mbox{if}\ \ e\nsubseteq\Gamma,\\
 g-\varepsilon\nabla u_{h}\cdot{\bf n}_{e}\ \ \ & \mbox{if}\ \ e\subset\Gamma_{N},\\
 0\ \ \ & \mbox{if}\ \ e\subset\Gamma_{D},
\end{array}\right.
\end{equation*}
respectively. A global upper bound is then given by \cite[Theorem 1]{John-Novo}
\begin{align}
    \nonumber
    ||u-u_{h}||_{\rm {SUPG}}^{2} \leq
    & \eta_{1}^{2}+\eta_{2}^{2}+\eta_{3}^{2}+\sum\limits_{K\in\mathcal{T}_{h}}16\delta_{K} h_{K}^{-2} \varepsilon^{2}C^{2} ||\nabla(u-\tilde{I}_{h}u_{h})||_{K}^{2}\\
    \label{Eq:GUBJN}
    & +\sum\limits_{K\in\mathcal{T}_{h}}8\delta_{K}\varepsilon^{2}||\triangle(u-\tilde{I}_{h}u_{h})||_{K}^{2},
\end{align}
where $||u-u_{h}||_{{\rm SUPG}}^{2}= ||u-u_{h}||_{\varepsilon}^{2}+\sum_{K\in\mathcal{T}_{h}}\delta_{K} ||{\bf a}\cdot\nabla(u-u_{h})||_{K}^{2}$ and $\tilde{I}_{h}$ is an interpolation operator satisfying the hypothesis in \cite{John-Novo}. In the convection-dominated regime, the last two terms in \eqref{Eq:GUBJN} are negligible compared with the other terms. The upper bound is reduced to
\begin{equation*}
||u-u_{h}||_{\rm {SUPG}}^{2} \preceq \eta_{1}^{2}+\eta_{2}^{2}+\eta_{3}^{2}.
\end{equation*}
Compared with the estimator in \cite{Du}, one concludes that
\begin{equation*}
|||u-u_{h}|||_{*}^{2}\preceq\eta_{1}^{2}+\eta_{2}^{2}+\eta_{3}^{2}+{\rm h.o.t}.
\end{equation*}
On the other hand, when convection dominates, the local lower bound is \cite[Theorem 2]{John-Novo}
\begin{equation*}
\eta_{i}\preceq||u-u_{h}||_{\rm{SUPG}}+{\rm h.o.t.}, \quad i=1,2,3.
\end{equation*}
This leads to
\begin{equation*}
|||u-u_{h}|||_{*}\preceq||u-u_{h}||_{\rm{SUPG}}+{\rm h.o.t}.
\end{equation*}

Let $|||u-u_{h}|||_{\varepsilon} := ||u-u_{h}||_{\varepsilon}+||h^{1/2}\nabla(u-u_{h})||$. Since
\begin{equation*}
||u-u_{h}||_{\rm{SUPG}} \preceq |||u-u_{h}|||_{\varepsilon} \leq |||u-u||| \preceq |||u-u_{h}|||_{*},
\end{equation*}
$|||u-u_{h}|||_{*}$ is equivalent to $|||u-u_{h}|||_{\varepsilon}$ and $||u-u_{h}||_{\rm{SUPG}}$ when the higher order terms are negligible. This will be confirmed numerically in Section~\ref{NumExp}.
\end{remark}

\section{Flux Recovery}\label{FluxRec}
Introducing the flux variable $\sigma=-\varepsilon\nabla u$, the variational form of the flux
reads: find $\sigma\in H({\rm div};\Omega)$ such that
\begin{equation}\label{rocover1}
(\varepsilon^{-1}\sigma,\tau)=-(\nabla u,\tau)\quad \forall \tau\in H({\rm div};\Omega).
\end{equation}
In this paper, we use standard $RT_{0}$ or $BDM_{1}$ elements to recover the flux, which are
\begin{align*}
    & RT_{0}: =\{\tau\in H({\rm div};\Omega): \tau|_{K}\in P_{0}(K)^{d}+{\bf x}P_{0}(K) \quad \forall K\in\mathcal{T}_{h}\} \\
    \text{and} \hspace{0.5in}
    & BDM_{1}: =\{\tau\in H({\rm div};\Omega): \tau|_{K}\in P_{1}(K)^{d} \quad \forall K\in\mathcal{T}_{h}\},
\end{align*}
respectively. Let $u_{h}$ be the solution to (\ref{PDE6}) and $\mathcal{V}$ be $RT_{0}$ or $BDM_{1}$. We recover the flux by
solving the following problem: find $\sigma_{\nu}\in\mathcal{V}$ such that
\begin{equation}\label{recover2}
    (\varepsilon^{-1}\sigma_{\nu},\tau)=-(\nabla u_{h},\tau)\quad \forall\tau\in\mathcal{V}.
\end{equation}
We have the following {\it a priori} error estimates for the recovered flux.

\begin{theorem}
Let $u$, $u_h$, $\sigma$, and $\sigma_{\nu}$ be solutions to (\ref{PDE3}), (\ref{PDE6}), \eqref{rocover1},
and (\ref{recover2}), respectively. Then
there holds
\begin{equation*}
    ||\varepsilon^{-1/2}(\sigma-\sigma_{\nu})|| \preceq \inf_{\tau\in\mathcal{V}}||\varepsilon^{-1/2}(\sigma-\tau)||
    + ||\varepsilon^{1/2}\nabla(u-u_{h})||.
\end{equation*}
\end{theorem}
\begin{proof}
Following the line of the proof of \cite[Theorem 3.1]{cai1}, we obtain the assertion.
\end{proof}

We next consider an explicit approximation of the flux in $RT_{0}$ (cf., e.g. \cite{cai1}).
For $e\in\mathcal{E}_{D} \cup \mathcal{E}_{N}$, let ${\bf n}_{e}$ be the outward unit normal vector
to $\Gamma$. For $e\in\mathcal{E}_{\Omega}$, let $K_{e}^{+}$ and $K_{e}^{-}$ be the two elements
sharing $e$, and let ${\bf n}_{e}$ be the outward unit normal vector of $K_{e}^{+}$. Let ${\bf a}_{e}^{\pm}$
be the opposite vertices of $e$ in $K_{e}^{\pm}$, respectively. Then the $RT_{0}$ basis function corresponding to $e$ is
\begin{equation*}
    \phi_{e}({\bf x}) :=\left \{ \begin{array}{ll}
    \displaystyle \frac{|e|}{d|K_{e}^{+}|}({\bf x}-{\bf a}_{e}^{+})\ \ \ \ \  & \mbox{for}\;  \ \ \ {\bf x}\in K_{e}^{+},\\
    \displaystyle -\frac{|e|}{d|K_{e}^{-}|}({\bf x}-{\bf a}_{e}^{-})\ \ \ \ \  & \mbox{for}\;  \ \ \ {\bf x}\in K_{e}^{-},\\
    0\ \ \ \ \  & \mbox{elsewhere},
    \end{array}\right.
\end{equation*}
where $|e|$ and $|K_{e}^{\pm}|$ are the $(d-1)$- and $d$-dimensional measures of $e$ and $K_{e}^{\pm}$, respectively. For
a boundary edge/face $e$, the corresponding basis function is
\begin{equation*}
    \phi_{e}({\bf x}) :=\left \{ \begin{array}{ll}
    \displaystyle \frac{|e|}{d|K_{e}^{+}|}({\bf x}-{\bf a}_{e}^{+})\ \ \ \ \  & \mbox{for}\;  \ \ \ {\bf x}\in K_{e}^{+},\\
    0\ \ \ \ \  & \mbox{elsewhere}.
    \end{array}\right.
\end{equation*}
Define the approximation $\hat{\sigma}_{RT_{0}}(u_{h})$ of $\boldsymbol{\tau}=-\varepsilon\nabla u_{h}$ in $RT_{0}$ by
\begin{equation}\label{recover4}
    \hat{\sigma}_{RT_{0}}(u_{h}) = \sum_{e\in\mathcal{E}} \hat{\sigma}_{e}\phi_{e} ({\bf x}),
\end{equation}
where $\hat{\sigma}_{e}$ is the normal component of $\hat{\sigma}_{RT_{0}}$ on
$e\in\mathcal{E}$ defined by
\begin{equation}\label{add2}
    \hat{\sigma}_{e} :=\left \{
    \begin{array}{ll}
        \gamma_{e}(\boldsymbol{\tau}|_{K_{e}^{+}}\cdot{\bf n}_{e})|_{e} +
        (1-\gamma_{e})(\boldsymbol{\tau}|_{K_{e}^{-}}\cdot{\bf n}_{e})|_{e} & \text{for }~ e\in\mathcal{E}_{\Omega},\\
        (\boldsymbol{\tau}|_{K_{e}^{+}}\cdot{\bf n}_{e})|_e & \text{for }~  e \in \mathcal{E}_{D}\cup\mathcal{E}_{N},
 \end{array}\right.
\end{equation}
with the constant $\gamma_{e}\in[0,1)$ to be determined in \eqref{Eq:DefnGm}. Note that the definition of
$\hat{\sigma}_{RT_{0}}(u_{h})$ is independent of
the choice of $K_{e}^{+}$ and $K_{e}^{-}$.

\section{{\em A posteriori} Error Estimates}\label{APostErr}
For $K\in\mathcal{T}_{h}$ and $e\in\mathcal{E}$, define weights $\alpha_{K} :=\min \big\{h_{K}\varepsilon^{-1/2}, \beta^{-1/2}, h_{K}^{1/2}\big\}$ and $\alpha_{e} := \min \big\{h_{e}^{1/2}\varepsilon^{-1/2}, \varepsilon^{-1/4}\beta^{-1/4}, 1\big\}$, and residuals
\begin{equation}
    \label{Eq:CellResidl}
    R_{K} :=f+\varepsilon\triangle u_{h}-{\bf a}\cdot\nabla u_{h}-bu_{h} \quad \text{and} \quad \tilde{R}_{K} :=f-
    \nabla\cdot\sigma_{h}-{\bf a}\cdot\nabla u_{h}-bu_{h},
\end{equation}
where $\sigma_{h}$ is the implicit or explicit recovered flux. Let
\begin{equation}\label{Eq:Phi}
    \Phi = \Big( \sum \limits_{K\in\mathcal{T}_{h}} \alpha_{K}^{2} (||R_{K}||_{K}^{2}+
    ||\tilde{R}_{K}||_{K}^{2}) + ||\varepsilon^{1/2}\nabla u_{h}+\varepsilon^{-1/2}\sigma_{h}||^{2} \Big)^{1/2}.
\end{equation}
We have the following error estimates.

\begin{theorem}\label{recover3}
Let $u$ and $u_{h}$ be the solutions to (\ref{PDE3}) and (\ref{PDE6}), respectively. Let $\Phi$ be defined in \eqref{Eq:Phi}. If $\sigma_{h} = \hat{\sigma}_{RT_{0}} (u_{h})$ is the recovered flux obtained by the explicit approximation (\ref{recover4}), then
\begin{equation}\label{recover5}
    |||u-u_{h}|||_{*} \preceq \Phi + \Big( \sum_{e\subset\Gamma_{N}} \alpha_{e}^{2}||g-\varepsilon\nabla u_{h}\cdot{\bf n}||_{e}^{2} \Big)^{1/2}.
\end{equation}
If $\sigma_{h}=\sigma_{\nu}$ is the recovered flux obtained by the implicit approximation (\ref{recover2}), then
\begin{equation}\label{aaa4}
    |||u-u_{h}|||_{*} \preceq  \Phi + \Big( \sum_{e\subset\Gamma_{N}} \alpha_{e}^{2} (||g-\varepsilon\nabla u_{h}\cdot{\bf n}||_{e}^{2} + ||(\sigma_{h} +
    \varepsilon\nabla u_{h}) \cdot{\bf n}||_{e}^{2}) \Big)^{1/2}.
\end{equation}
\end{theorem}
\begin{proof} For any $v\in H_{D}^{1}(\Omega)$, let $I_{h}v$ be the Cl\'{e}ment interpolation of $v$. Using \eqref{PDE3}, and
integration by parts, we have
\begin{align*}
    B(u-u_{h},v) & = (f-{\bf a}\cdot\nabla u_{h}-bu_{h},v) - (\varepsilon\nabla u_{h},\nabla v)+<g,v>_{\Gamma_{N}} \\
    & = (\tilde{R}_{K},v) - (\varepsilon^{1/2}\nabla u_{h}+\varepsilon^{-1/2}\sigma_{h},\varepsilon^{1/2}\nabla v) +
    <g+\sigma_{h}\cdot{\bf n},v>_{\Gamma_{N}},
\end{align*}
which implies
\begin{align}
    \nonumber
    B(u-u_{h},v-I_{h}v) = & -\sum_{K\in\mathcal{T}_{h}}(\varepsilon^{1/2}\nabla u_{h}+
    \varepsilon^{-1/2}\sigma_{h},\varepsilon^{1/2}\nabla(v-I_{h}v))_{K} \\
    \label{dominated 2}
    & +\sum_{K\in\mathcal{T}_{h}}(\tilde{R}_{K},v-I_{h}v)_{K}+\sum_{e\subset\Gamma_{N}}<g+\sigma_{h}\cdot{\bf n},v-I_{h}v>_{e}.
\end{align}
Subtracting \eqref{PDE6} from \eqref{PDE3}, we get
\begin{equation}\label{dominated 3}
B(u-u_{h},I_{h}v)=-\sum_{K\in\mathcal{T}_{h}}\delta_{K}(R_{K},{\bf
a}\cdot\nabla(I_{h}v))_{K}.
\end{equation}
On the other hand, the Cl\'{e}ment interpolation operator possesses the following
stable estimate (cf. \cite[Exercise 3.2.3]{Ciarlet} and \cite{Clement,Verfurth1,San08})
\begin{equation*}
||\nabla(v-I_{h}v)||_{K}\preceq||\nabla v||_{\tilde{\omega}_{K}} \quad \forall K\in\mathcal{T}_{h}, \quad v\in H^{1}(\tilde{\omega}_{K}),
\end{equation*}
where $\tilde{\omega}_{K}$ is the union of all elements in $\mathcal{T}_{h}$ sharing at least
one point with $K$. Then from (\ref{dominated 2}), (\ref{dominated 3}), and Lemma~\ref{convection 1}, we obtain
\begin{align}
    \nonumber
    B(u-u_{h}, &v)=B(u-u_{h},v-I_{h}v)+B(u-u_{h},I_{h}v)\\
    \nonumber
    \preceq \Big( & \sum_{K\in\mathcal{T}_{h}}\max\{\beta,\varepsilon h_{K}^{-2},h_{K}^{-1}\}^{-1}(||R_{K}||_{K}^{2}+
    ||\tilde{R}_{K}||_{K}^{2})+||\varepsilon^{1/2}\nabla u_{h}+\varepsilon^{-1/2}\sigma_{h}||^{2}\\
    \label{dominated 4}
    & + \sum_{e\subset\Gamma_{N}}\max\{\varepsilon^{1/2} \beta^{1/2},\varepsilon h_{e}^{-1},1\}^{-1}||g+
    \sigma_{h}\cdot{\bf n}||_{e}^{2} \Big)^{1/2}|||v|||.
\end{align}
If $\sigma_{h}$ is the recovery flux obtained by its explicit
approximation (\ref{recover4}), i.e., $\sigma_{h} = \hat{\sigma}_{RT_{0}} (u_{h})$, then we have from the
construction of $\hat{\sigma}_{RT_{0}} (u_{h})$ that
\begin{equation*}
\sigma_{h}\cdot{\bf n}=-\varepsilon\nabla u_{h}\cdot{\bf n}\ \ \ {\rm on}\ \ \Gamma_{N}.
\end{equation*}
Thus (\ref{recover5}) follows from (\ref{dominated 4}). If $\sigma_{h}$ is the recovery flux obtained by the implicit
approximation (\ref{recover2}), i.e., $\sigma_{h}=\sigma_{\nu}$, then (\ref{aaa4}) follows from a triangle
inequality and (\ref{dominated 4}).
\end{proof}

\begin{figure}[t]
    \centering
    \includegraphics[width=5in]{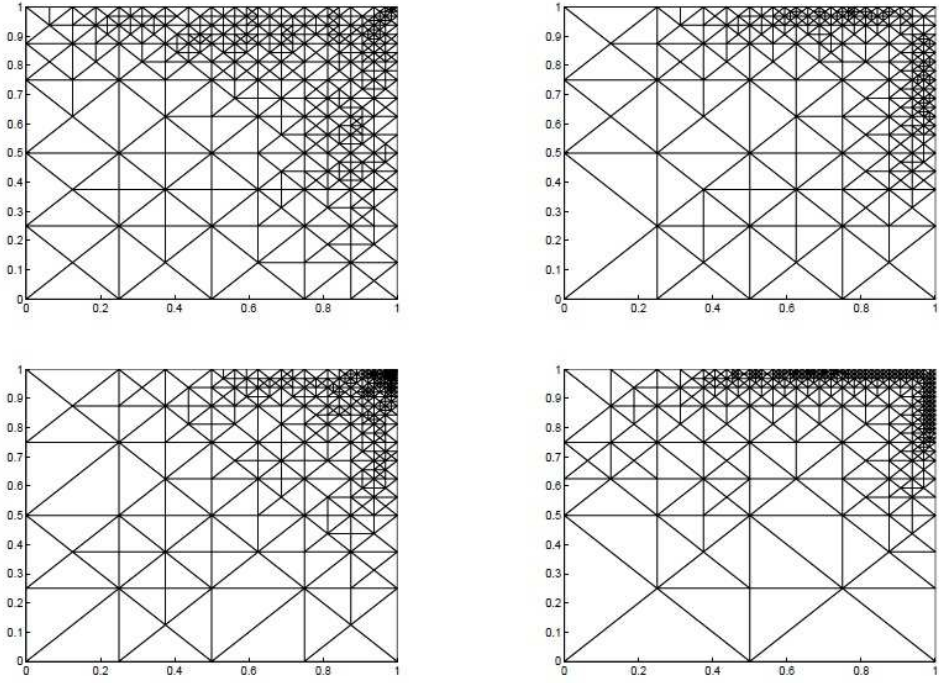}
    \caption{Top: the meshes by using $\varepsilon=0.01$, $\delta_{K}=h_{K}$, and $\theta=0.3$; Bottom: the meshes by using $\varepsilon=0.0001$, $\delta_{K}=h_{K}$, and $\theta=0.5$; Left: the meshes by using $\big(\sum_{e\in\mathcal{E}_{h}}\alpha_{e}^{2} ||R_{e}||_{e}^{2}\big)^{1/2}$; and Right: the meshes by using $||\varepsilon^{-1/2}\sigma_{h}+\varepsilon^{1/2}\nabla u_{h}||$ with $\sigma_{h}$ the implicit recovery flux by (\ref{recover2}). The top-left, top-right, bottom-left, and bottom-right plots are meshes after 10, 8, 8, and 8 iterations with 640, 473, 635, and 749 triangles, respectively.}
    \label{aaaabb1}
\end{figure}

\begin{remark}
Compared with the estimators developed in \cite{Du}, the jump in the
normal component of the flux is replaced by the residual
$||\varepsilon^{1/2}\nabla u_{h} + \varepsilon^{-1/2} \sigma_{h}||$
in Theorem~\ref{recover3}. Moreover, a residual term
$\sum_{K\in\mathcal{T}_{h}} \alpha_{K}^{2}
||\tilde{R}_{K}||_{K}^{2}$ and another residual term of the Neumann
boundary data occur in the {\em a posteriori} error estimators. In
$||\varepsilon^{1/2}\nabla u_{h}+\varepsilon^{-1/2}\sigma_{h}||$,
the impacts of $\varepsilon$ and $h$ are implicitly accounted, which
however are expressed explicitly in $\alpha_{e}$, and hence in
$\big(\sum_{e\in\mathcal{E}_{h}}
\alpha_{e}^{2}||R_{e}||_{e}^{2}\big)^{1/2}$, in \cite{Du} (e.g., if
$\varepsilon\leq h_{e}$ and $\beta=1$, then $\alpha_{e}=1$). To
illustrate the difference in numerical results, we provide in
Figure \ref{aaaabb1} the adaptive meshes by the two estimators for Section 7 Example 1. It is observed that the
quality of the meshes generated by $||\varepsilon^{1/2}\nabla u_{h}
+ \varepsilon^{-1/2} \sigma_{h}||$ is better than that of the meshes
by $\big(\sum_{e\in\mathcal{E}_{h}}\alpha_{e}^{2}
||R_{e}||_{e}^{2}\big)^{1/2}$.
\end{remark}

\section{Analysis of efficiency}\label{Sec_analysis}
Let $\boldsymbol{\tau}=-\varepsilon\nabla u_{h}$. For each $e\in\mathcal{E}_{\Omega}$, define the edge/face residual along $e$ by
\begin{equation*}
    R_{e} :=\left \{
    \begin{array}{ll}
        J_{e}(\boldsymbol{\tau}) & \mbox{if}\ \ e\nsubseteq\Gamma,\\
        g+\boldsymbol{\tau}\cdot{\bf n}_{e} & \mbox{if}\ \ e\subset\Gamma_{N},\\
        0 & \mbox{if}\ \ e\subset\Gamma_{D},
    \end{array}\right.
\end{equation*}
where $J_{e}(\boldsymbol{\tau})$ is defined in \eqref{Eq:DefnJe}. Let $\Pi_{k}$ be an $L^{2}$-projection operator into $P_{k}(K)$ and
\begin{equation*}
    {\rm osc}_{h}:=\Big(\sum_{K\in\mathcal{T}_{h}}\alpha_{K}^{2}||D_{K}||_{K}^{2} +
    \sum_{e\subset\Gamma_{N}}\alpha_{e}^{2}||D_{e}||_{e}^{2} \Big)^{1/2},
\end{equation*}
be an oscillation of data, where $D_{K}=R_{K}-\Pi_{k}R_{K} $ for every $K\in\mathcal{T}_{h}$,
and $D_{e}=R_{e}-\Pi_{k}R_{e}$ for each $e\subset\Gamma_{N}$. The following efficient estimate is found in \cite{Du}.

\begin{lemma}\label{recover15}
Let $u$ and $u_{h}$ be the solutions to problems (\ref{PDE3}) and (\ref{PDE6}), respectively. Then the error is bounded from below by
\begin{equation}
    \Big(\sum_{K\in\mathcal{T}_{h}} \big(\alpha_{K}^{2}||R_{K}||_{K}^{2} +
    \sum_{e\subset\partial K} \alpha_{e}^{2}||R_{e}||_{e}^{2} \big) \Big)^{1/2} \preceq |||u-u_{h}|||_{*}+{\rm osc}_{h}.
\end{equation}
\end{lemma}

\begin{lemma}\label{recover6}
Let $u$ and $u_{h}$ be the solutions to (\ref{PDE3}) and (\ref{PDE6}), respectively. If $\hat{\sigma}_{RT_{0}}(u_{h})$ is the
explicit recovery flux given by (\ref{recover4}), then it holds
\begin{equation}\label{efficiency1}
    ||\varepsilon^{-1/2}\hat{\sigma}_{RT_{0}}(u_{h})+\varepsilon^{1/2}\nabla u_{h}||\preceq |||u-u_{h}|||_{*}+{\rm osc}_{h}.
\end{equation}
\end{lemma}
\begin{proof}
For any element $K\in\mathcal{T}_{h}$ and an edge/face $e\subset\partial K$, let ${\bf n}_{e}$ be the outward unit vector normal to $\partial K$. Note that $\boldsymbol{\tau}=-\varepsilon\nabla u_{h}$ on $K$ is a constant vector. Let $\boldsymbol{\tau}_{e,K} = (\boldsymbol{\tau}|_{K} \cdot{\bf n}_{e})|_{e}$ be the normal component of $\boldsymbol{\tau}$ on $e$. There holds the representation in $RT_{0}$: $\boldsymbol{\tau} = \sum_{e\subset\partial K} \boldsymbol{\tau}_{e,K} \phi_{e}({\bf x})$. Then, for ${\bf x}\in K$, (\ref{recover4}) and (\ref{add2}) give
\begin{equation*}
    \hat{\sigma}_{RT_{0}}(u_{h})-\boldsymbol{\tau} = \sum_{e\subset\partial K\cap\mathcal{E}_{\Omega}} (\hat{\sigma}_{e}-
    \boldsymbol{\tau}|_{e}\cdot{\bf n}_{e}) \phi_{e}({\bf x})
    =\sum_{e\subset\partial K\setminus\Gamma}(1-\gamma_{e})J_{e}(\boldsymbol{\tau})\phi_{e}({\bf x}),
\end{equation*}
where, for the two elements $K_{e}^{+}$ and $K_{e}^{-}$ sharing $e$,
\begin{equation}\label{Eq:DefnJe}
    J_{e}(\boldsymbol{\tau})=(\boldsymbol{\tau}|_{K_{e}^{+}}-\boldsymbol{\tau}|_{K_{e}^{-}})\cdot{\bf n}_{e}.
\end{equation}
This identity implies
\begin{align}
    \nonumber
    & ||\varepsilon^{-1/2}(\hat{\sigma}_{RT_{0}}(u_{h})-\boldsymbol{\tau})||_{K}^{2}
    \preceq \sum_{e\subset\partial K\setminus\Gamma}\frac{(1-\gamma_{e})^{2}}{\varepsilon}||J_{e}(\boldsymbol{\tau})\phi_{e}({\bf x})||_{K}^{2} \\
    \label{duaddition}
    \leq & \sum_{e\subset\partial K\setminus\Gamma}\frac{(1-\gamma_{e})^{2}}{\varepsilon}|J_{e}(\boldsymbol{\tau})|^{2}||\phi_{e}({\bf x})||_{K}^{2}
    \preceq \sum_{e\subset\partial K\setminus\Gamma}\frac{(1-\gamma_{e})^{2}}{\varepsilon}||J_{e}(\boldsymbol{\tau})||_{e}^{2}h_{e},
\end{align}
where, in the last step, we employ the fact that $J_{e}(\boldsymbol{\tau})$ is constant and $||\phi_{e}({\bf x})||_{K}^{2}\preceq|K|$.

Now, for each $e\in\mathcal{E}_{\Omega}$ we choose
\begin{equation}\label{Eq:DefnGm}
    \gamma_{e}=1-\alpha_{e}\varepsilon^{1/2}h_{e}^{-1/2}
\end{equation}
so that ${(1-\gamma_{e})^{2}h_{e}} / {\varepsilon} = 1$. Since $\alpha_{e}\leq\sqrt{h_{e} / \varepsilon}$, thus $0\leq\gamma_{e}<1$. This choice together with
the definition of the edge$/$face residual
$R_{e}$ leads to
\begin{equation*}
    \frac{(1-\gamma_{e})^{2}}{\varepsilon}h_{e}||J_{e}(\boldsymbol{\tau})||_{e}^{2} \leq
    \alpha_{e}^{2}||R_{e}||_{e}^{2},
\end{equation*}
which, with (\ref{duaddition}), implies
\begin{equation}\label{recover10}
||\varepsilon^{-1/2}\hat{\sigma}_{RT_{0}}(u_{h}) + \varepsilon^{1/2}\nabla u_{h}||_{K}^{2}
\preceq\displaystyle\sum_{e\subset\partial K\setminus\Gamma}\alpha_{e}^{2}||R_{e}||_{e}^{2}.
\end{equation}
Summing up (\ref{recover10}) over all $K\in\mathcal{T}_{h}$, we obtain
\begin{equation}\label{recover11}
    ||\varepsilon^{-1/2}\hat{\sigma}_{RT_{0}}(u_{h})+\varepsilon^{1/2}\nabla u_{h}||^{2}
    \preceq \sum_{K\in\mathcal{T}_{h}}\sum_{e\subset\partial K\setminus\Gamma}\alpha_{e}^{2}||R_{e}||_{e}^{2}
    \preceq \sum_{K\in\mathcal{T}_{h}}\sum_{e\subset\partial K}\alpha_{e}^{2}||R_{e}||_{e}^{2}.
\end{equation}
The desired estimate (\ref{efficiency1}) follows from (\ref{recover11}) and (\ref{recover15}).
\end{proof}

\begin{lemma}\label{yingli5.2}
Under the assumption of Lemma \ref{recover6}, if $\sigma_{\nu}$ is the implicit recovery
flux obtained by (\ref{recover2}), then it holds
\begin{equation}\label{yingli5.1}
    ||\varepsilon^{-1/2}\sigma_{\nu}+\varepsilon^{1/2}\nabla u_{h}||\preceq|||u-u_{h}|||_{*}+{\rm osc}_{h}.
\end{equation}
\end{lemma}
\begin{proof}
For all $\tau\in\mathcal{V}$, (\ref{recover2}) implies
\begin{equation*}
(\varepsilon^{-1/2}\sigma_{\nu}+\varepsilon^{1/2}\nabla u_{h},\varepsilon^{-1/2}\tau)=0,
\end{equation*}
which results in
\begin{align*}
    & ||\varepsilon^{-1/2}\sigma_{\nu}+\varepsilon^{1/2}\nabla u_{h}||^{2} = (\varepsilon^{-\frac{1}{2}}\sigma_{\nu}+
    \varepsilon^{\frac{1}{2}}\nabla u_{h},\varepsilon^{-\frac{1}{2}}(\sigma_{\nu}-\tau)+
    \varepsilon^{-\frac{1}{2}}\tau+\varepsilon^{\frac{1}{2}}\nabla u_{h})\\
    = & (\varepsilon^{-1/2}\sigma_{\nu}+\varepsilon^{1/2}\nabla u_{h},\varepsilon^{-1/2}\tau+\varepsilon^{1/2}\nabla u_{h})
    \leq ||\varepsilon^{-1/2}\sigma_{\nu}+\varepsilon^{1/2}\nabla u_{h}||||\varepsilon^{-1/2}\tau+\varepsilon^{1/2}\nabla u_{h}||.
\end{align*}
Dividing by $||\varepsilon^{-1/2}\sigma_{\nu}+\varepsilon^{1/2}\nabla u_{h}||$, we get
\begin{equation*}
||\varepsilon^{-1/2}\sigma_{\nu}+\varepsilon^{1/2}\nabla u_{h}||\leq ||\varepsilon^{-1/2}\tau+\varepsilon^{1/2}\nabla u_{h}||
\end{equation*}
for all $\tau\in\mathcal{V}$, which implies
\begin{equation}\label{recover14}
    ||\varepsilon^{-1/2}\sigma_{\nu}+\varepsilon^{1/2}\nabla u_{h}||= \min_{\tau\in\mathcal{V}}||\varepsilon^{-1/2}\tau+\varepsilon^{1/2}\nabla u_{h}||.
\end{equation}
The assertion (\ref{yingli5.1}) follows from the fact that $RT_{0}\subset BDM_{1}$, (\ref{recover14}), and Lemma \ref{recover6}.
\end{proof}

\begin{lemma}\label{aaa5}
Under the assumption of Lemma \ref{recover6}, if $\sigma_{\nu}$ is the implicit recovery flux obtained by (\ref{recover2}), then it holds
\begin{equation}\label{aaa6}
    \Big( \sum_{e\subset\Gamma_{N}}\alpha_{e}^{2} ||(\sigma_{\nu}+
    \varepsilon\nabla u_{h}) \cdot{\bf n}||_{e}^{2} \Big)^{1/2} \preceq ||\varepsilon^{-1/2}\sigma_{\nu} + \varepsilon^{1/2}\nabla u_{h}||.
\end{equation}
\end{lemma}
\begin{proof}
Using trace theorem, inverse estimate, shape regularity of element, and the
fact $\alpha_{e}\leq h_{e}^{1/2}/\sqrt{\varepsilon}$, we have for $e\subset\Gamma_{N} \cap \partial K$
\begin{equation*}
    \alpha_{e} ||(\sigma_{\nu}+\varepsilon\nabla u_{h})\cdot{\bf n}||_{e}
    \preceq \alpha_{e}h_{K}^{-1/2} ||\sigma_{\nu}+\varepsilon\nabla u_{h}||_{K}
    \preceq ||\varepsilon^{-1/2}\sigma_{\nu}+\varepsilon^{1/2}\nabla u_{h}||_{K}.
\end{equation*}
Summing the above inequality over all $e\subset\Gamma_{N}$, we obtain the desired estimate (\ref{aaa6}).
\end{proof}

Moreover, we have the following estimate.
\begin{lemma}\label{recover16}
Let $\sigma_{h}$ be the flux recovery obtained by the implicit approximation (\ref{recover2}) or the
explicit approximation (\ref{recover4}). Then it holds
\begin{equation}\label{recover17}
    \Big(\sum_{K\in\mathcal{T}_{h}}\alpha_{K}^{2}||\tilde{R}_{K}||_{K}^{2} \Big)^{1/2} \preceq
     \Big(\sum_{K\in\mathcal{T}_{h}}\alpha_{K}^{2}||R_{K}||_{K}^{2} \Big)^{1/2} + ||\varepsilon^{1/2}\nabla u_{h}+\varepsilon^{-1/2}\sigma_{h}||.
\end{equation}
\end{lemma}
\begin{proof}
For each $K\in\mathcal{T}_{h}$, it follows from triangle inequality and inverse estimate that
\begin{equation*}
\begin{array}{lll}
||\tilde{R}_{K}||_{K}&\leq&||R_{K}||_{K}+||\varepsilon\triangle u_{h}+\nabla\cdot\sigma_{h}||_{K}\\
&\preceq&||R_{K}||_{K}+h_{K}^{-1}\varepsilon^{1/2}||\varepsilon^{1/2}\nabla u_{h}+\varepsilon^{-1/2}\sigma_{h}||_{K}.
\end{array}
\end{equation*}
We get from the fact $\alpha_{K}\leq h_{K}/\sqrt{\varepsilon}$ that
\begin{equation*}
\alpha_{K}||\tilde{R}_{K}||_{K}\preceq\alpha_{K}||R_{K}||_{K}+||\varepsilon^{1/2}\nabla u_{h}+
\varepsilon^{-1/2}\sigma_{h}||_{K}.
\end{equation*}
Summing up the above inequality over all $K\in\mathcal{T}_{h}$, we obtain
\begin{equation*}
\displaystyle\sum_{K\in\mathcal{T}_{h}}\alpha_{K}^{2}||\tilde{R}_{K}||_{K}^{2}\preceq
\sum_{K\in\mathcal{T}_{h}}\alpha_{K}^{2}||R_{K}||_{K}^{2}+||\varepsilon^{1/2}\nabla u_{h}+
\varepsilon^{-1/2}\sigma_{h}||^{2},
\end{equation*}
which results in the desired estimate (\ref{recover17}).
\end{proof}

Collecting Lemma \ref{recover15}-\ref{recover16}, we obtain the global lower bound estimate.
\begin{theorem}\label{add4}
Let $u$ and $u_{h}$ be the solutions to (\ref{PDE3}) and (\ref{PDE6}), respectively. Let $\Phi$ be defined in \eqref{Eq:Phi}.
If $\sigma_{h}$ is the recovery flux obtained by the explicit approximation (\ref{recover4}), i.e.,
$\sigma_{h}=\hat{\sigma}_{RT_{0}}(u_{h})$, then
\begin{equation} \label{recover18}
    \Phi + \Big(\sum_{e\subset\Gamma_{N}} \alpha_{e}^{2} ||g-\varepsilon\nabla u_{h}\cdot{\bf n}||_{e}^{2} \Big)^{1/2}
    \preceq |||u-u_{h}|||_{*} + {\rm osc}_{h}.
\end{equation}
If $\sigma_{h}$ is the recovery flux obtained by the implicit approximation (\ref{recover2}), i.e., $\sigma_{h}=\sigma_{\nu}$,
then
\begin{equation}
    \label{aaa7}
    \Phi + \Big(\sum_{e\subset\Gamma_{N}}\alpha_{e}^{2}(||g-\varepsilon\nabla u_{h}\cdot{\bf n}||_{e}^{2} +
    ||(\sigma_{h} + \varepsilon\nabla u_{h})\cdot{\bf n}||_{e}^{2}) \Big)^{\frac{1}{2}}
    \preceq  |||u-u_{h}|||_{*}+{\rm osc}_{h}.
\end{equation}
\end{theorem}
\begin{proof}
It follows from Lemmas \ref{recover6}-\ref{yingli5.2} that
\begin{equation}\label{recover19}
||\varepsilon^{1/2}\nabla u_{h}+\varepsilon^{-1/2}\sigma_{h}||^{2} \preceq |||u-u_{h}|||_{*}^{2} + {\rm osc}_{h}.
\end{equation}
If $\sigma_{h}=\hat{\sigma}_{RT_{0}}(u_{h})$, we get from Lemma~\ref{recover15} and Lemma~\ref{recover16} that
\begin{align}
    \nonumber
    & \sum_{K\in\mathcal{T}_{h}}\alpha_{K}^{2}(||R_{K}||_{K}^{2} + ||\tilde{R}_{K}||_{K}^{2})+
    \sum_{e\subset\Gamma_{N}}\alpha_{e}^{2}||g-\varepsilon\nabla u_{h}\cdot{\bf n}||_{e}^{2} \\
    \nonumber
    \leq & \sum_{K\in\mathcal{T}_{h}}(\alpha_{K}^{2}||R_{K}||_{K}^{2} +
    \sum_{e\subset\partial K}\alpha_{e}^{2} ||R_{e}||_{e}^{2}) + \sum_{K\in\mathcal{T}_{h}}\alpha_{K}^{2} ||\tilde{R}_{K}||_{K}^{2} \\
    \nonumber
    \preceq & \sum_{K\in\mathcal{T}_{h}}(\alpha_{K}^{2}||R_{K}||_{K}^{2} +
    \sum_{e\subset\partial K}\alpha_{e}^{2}||R_{e}||_{e}^{2}) + ||\varepsilon^{1/2}\nabla u_{h} + \varepsilon^{-1/2}\sigma_{h}||^{2} \\
    \label{recover20}
    \preceq & ~|||u-u_{h}|||_{*}^{2}+{\rm osc}_{h}^{2} + ||\varepsilon^{1/2}\nabla u_{h}+\varepsilon^{-1/2}\sigma_{h}||^{2}.
\end{align}

The assertion (\ref{recover18}) follows from a combination of (\ref{recover19}) and (\ref{recover20}).
If $\sigma_{h}=\sigma_{\nu}$, then, similarly, we have from Lemma~\ref{recover15} and Lemmas~\ref{aaa5}-\ref{recover16} that
\begin{align*}
    & \sum_{K\in\mathcal{T}_{h}}\alpha_{K}^{2}(||R_{K}||_{K}^{2} + ||\tilde{R}_{K}||_{K}^{2}) +
    \sum_{e\subset\Gamma_{N}}\alpha_{e}^{2} (||g-\varepsilon\nabla u_{h}\cdot{\bf n}||_{e}^{2} +
    ||(\sigma_{h}+\varepsilon\nabla u_{h})\cdot{\bf n}||_{e}^{2}) \\
    \preceq & ~ |||u-u_{h}|||_{*}^{2} + {\rm osc}_{h}^{2}+||\varepsilon^{1/2}\nabla u_{h}+\varepsilon^{-1/2}\sigma_{h}||^{2}.
\end{align*}
The estimate (\ref{aaa7}) follows from the above inequality and (\ref{recover19}).
\end{proof}

\section{A stabilization $H({\rm div})$ recovery}\label{HdivRec}
Let $u_{h}\in V_{h}$ be the approximation of the solution $u$ to (\ref{PDE1}). A stabilization $H({\rm div})$
recovery procedure is to find $\sigma_{\mathcal{T}}\in\mathcal{V}$ such that
\begin{align}
    \nonumber
    & (\varepsilon^{-1}\sigma_{\mathcal{T}},\tau_{\nu}) +
    \sum_{K\in\mathcal{T}_{h}}\gamma_{K}(\nabla\cdot\sigma_{\mathcal{T}},\nabla\cdot\tau_{\nu})_{K} \\
    \label{recover21}
    = & -(\nabla u_{h},\tau_{\nu}) +
    \sum_{K\in\mathcal{T}_{h}}\gamma_{K} (f-{\bf a}\cdot\nabla u_{h}-bu_{h}, \nabla\cdot\tau_{\nu})_{K} \quad \forall \tau_{\nu} \in\mathcal{V},
\end{align}
where $\gamma_{K}$ is a stabilization parameter to be determined in below. Recalling the exact
flux $\sigma=-\varepsilon\nabla u$, define the approximation error of the flux recovery by
\begin{equation*}
||\sigma-\sigma_{\mathcal{T}}||_{B,\Omega}^{2} :=(\varepsilon^{-1}(\sigma-\sigma_{\mathcal{T}}),\sigma-\sigma_{\mathcal{T}})+
\displaystyle\sum_{K\in\mathcal{T}_{h}}
\gamma_{K}(\nabla\cdot(\sigma-\sigma_{\mathcal{T}}),\nabla\cdot(\sigma-\sigma_{\mathcal{T}}))_{K}.
\end{equation*}

\begin{theorem}
The following {\em a priori} error bound for the approximation error of the $H({\rm div})$ recovery flux holds
\begin{equation}\label{recover22}
||\sigma-\sigma_{\mathcal{T}}||_{B,\Omega}\preceq\inf_{\tau_{\nu}\in\mathcal{V}}||\sigma-\tau_{\nu}||_{B,\Omega}+
||u-u_{h}||_\varepsilon.
\end{equation}
\end{theorem}
\begin{proof}
Note that the exact flux $\sigma$ satisfies, for all $\tau\in H({\rm div};\Omega)$,
\begin{equation*}
    (\varepsilon^{-1}\sigma,\tau) + \displaystyle \sum_{K\in\mathcal{T}_{h}} \gamma_{K} (\nabla\cdot\sigma,\nabla\cdot\tau)_{K}
    = -(\nabla u,\tau)+\displaystyle\sum_{K\in\mathcal{T}_{h}}\gamma_{K}(f-{\bf a}\cdot\nabla u - bu,\nabla\cdot\tau)_{K}.
\end{equation*}
For all $\tau_{\nu}\in\mathcal{V}$, This identity and (\ref{recover21}) give the error equation
\begin{align}
    \nonumber
    & (\varepsilon^{-1}(\sigma-\sigma_{\mathcal{T}}),\tau_{\nu}) +
    \sum_{K\in\mathcal{T}_{h}} \gamma_{K} (\nabla\cdot(\sigma-\sigma_{\mathcal{T}}), \nabla\cdot\tau_{\nu})_{K} \\
    \label{recover24}
    = & -(\nabla(u-u_{h}),\tau_{\nu}) - \sum_{K\in\mathcal{T}_{h}}\gamma_{K}({\bf a}\cdot\nabla(u-u_{h})+b(u-u_{h}), \nabla\cdot\tau_{\nu})_{K}
\end{align}
which implies
\begin{align*}
    ||\sigma- & \sigma_{\mathcal{T}}||_{B,\Omega}^{2} =~
    (\varepsilon^{-1}(\sigma-\sigma_{\mathcal{T}}),\sigma-\tau_{\nu}) +
    \sum_{K\in\mathcal{T}_{h}}\gamma_{K} (\nabla\cdot(\sigma-\sigma_{\mathcal{T}}), \nabla\cdot(\sigma-\tau_{\nu}))_{K} \\
    & -(\nabla(u-u_{h}),\tau_{\nu}-\sigma_{\mathcal{T}}) -
    \sum_{K\in\mathcal{T}_{h}}\gamma_{K}({\bf a}\cdot\nabla(u-u_{h})+b(u-u_{h}), \nabla\cdot(\tau_{\nu}-\sigma_{\mathcal{T}}))_{K}.
\end{align*}
Using \eqref{Assumption} and Cauchy-Schwartz inequality, we arrive at
\begin{align*}
    ||\sigma- & \sigma_{\mathcal{T}}||_{B,\Omega}^{2} \leq ||\sigma-
    \sigma_{\mathcal{T}}||_{B,\Omega}||\sigma-\tau_{\nu}||_{B,\Omega} + ||u-u_{h}||_\varepsilon ||\tau_{\nu}-\sigma_{\mathcal{T}}||_{B,\Omega} \\
    & + \sum_{K\in\mathcal{T}_{h}}\gamma_{K} (||{\bf a}||_{L^{\infty}(K)}||\nabla(u-u_{h})||_{K}+
    c_{b}\beta||u-u_{h}||_{K}) ||\nabla\cdot(\tau_{\nu}-\sigma_{\mathcal{T}})||_{K}.
\end{align*}
Choose $\gamma_{K} \leq h_{K} \min \big\{\frac{1}{||{\bf a}||_{L^{\infty}(K)}},\frac{1}{\sqrt{\beta\varepsilon}}\big\}$ for
all $K\in\mathcal{T}_{h}$. Then, by inverse estimate, we have
\begin{align*}
    ||\sigma- & \sigma_{\mathcal{T}}||_{B,\Omega}^{2}
    \preceq ||\sigma-\sigma_{\mathcal{T}}||_{B,\Omega} ||\sigma-\tau_{\nu}||_{B,\Omega} +
    ||u-u_{h}||_\varepsilon ||\tau_{\nu}-\sigma_{\mathcal{T}}||_{B,\Omega} \\
    & \leq||\sigma-\sigma_{\mathcal{T}}||_{B,\Omega} ||\sigma-\tau_{\nu}||_{B,\Omega} +
    ||u-u_{h}||_\varepsilon (||\tau_{\nu}-\sigma||_{B,\Omega} + ||\sigma-\sigma_{\mathcal{T}}||_{B,\Omega}),
\end{align*}
which implies
\begin{equation*}
    ||\sigma-\sigma_{\mathcal{T}}||_{B,\Omega}\preceq||\sigma-\tau_{\nu}||_{B,\Omega}+
    ||u-u_{h}||_\varepsilon \quad \forall \tau_{\nu}\in\mathcal{V}.
\end{equation*}
The assertion (\ref{recover22}) follows immediately.
\end{proof}

\begin{theorem}\label{vvv}
Let $\sigma_{\mathcal{T}}$ be the $H({\rm div})$ recovery flux obtained from (\ref{recover21}),
and $u$ and $u_{h}$ be the solutions to (\ref{PDE3}) and (\ref{PDE6}), respectively. For
each $K\in\mathcal{T}_{h}$, let $\tilde{R}_{K} :=f-\nabla\cdot\sigma_{\mathcal{T}}-{\bf a}\cdot\nabla u_{h}-bu_{h}$.
Then the following reliable estimate holds
\begin{align}
    \nonumber
    |||u-u_{h}|||_{*}
    \preceq & \Big( \sum_{K\in\mathcal{T}_{h}}\alpha_{K}^{2}(||R_{K}||_{K}^{2} + ||\tilde{R}_{K}||_{K}^{2}) +
    ||\varepsilon^{1/2}\nabla u_{h}+\varepsilon^{-1/2}\sigma_{\mathcal{T}}||^{2}\Big)^{1/2} \\
    \label{recover26}
    & + \Big( \sum_{e\subset\Gamma_{N}} \alpha_{e}^{2}(||g-\varepsilon\nabla u_{h}\cdot{\bf n}||_{e}^{2} +
    ||(\sigma_{\mathcal{T}}+\varepsilon\nabla u_{h})\cdot{\bf n}||_{e}^{2}) \Big)^{1/2}.
\end{align}
\end{theorem}
\begin{proof}
Following the line of the proof of Theorem~\ref{recover3}, we obtain the estimate (\ref{recover26}).
\end{proof}

\begin{lemma}\label{aaa8}
Under the assumption of Lemma \ref{recover6}, if $\sigma_{\mathcal{T}}$ is the $H({\rm div})$ recovery flux
obtained from (\ref{recover21}), then it holds that
\begin{equation}\label{aaa9}
    \Big( \sum_{e\subset\Gamma_{N}}\alpha_{e}^{2} ||(\sigma_{\mathcal{T}} +
    \varepsilon\nabla u_{h})\cdot{\bf n}||_{e}^{2} \Big)^{1/2} \preceq ||\varepsilon^{-1/2}\sigma_{\mathcal{T}}+\varepsilon^{1/2}\nabla u_{h}||.
\end{equation}
\end{lemma}
\begin{proof}
A proof similar to Lemma \ref{aaa5} yields the assertion (\ref{aaa9}).
\end{proof}

\begin{lemma}\label{aaa10}
Let $\sigma_{\mathcal{T}}$ be the $H({\rm div})$ recovery flux obtained from (\ref{recover21}), and
$\tilde{R}_{K}$ be the residual defined in Theorem \ref{vvv}. Then it holds that
\begin{equation}\label{aaa11}
    \Big( \sum_{K\in\mathcal{T}_{h}}\alpha_{K}^{2}||\tilde{R}_{K}||_{K}^{2} \Big)^{1/2} \preceq
    \Big( \sum_{K\in\mathcal{T}_{h}}\alpha_{K}^{2}||R_{K}||_{K}^{2} \Big)^{1/2} +
    ||\varepsilon^{1/2}\nabla u_{h}+\varepsilon^{-1/2}\sigma_{\mathcal{T}}||.
\end{equation}
\end{lemma}
\begin{proof}
Following the line of the proof of Lemma \ref{recover16}, we obtain the assertion (\ref{aaa11}).
\end{proof}

\begin{lemma}\label{supplement}
Let $u$ and $u_{h}$ be the solutions to (\ref{PDE3}) and (\ref{PDE6}), respectively, and $\sigma_{\mathcal{T}}$ be the
$H(\rm div)$ recovery flux obtained from (\ref{recover21}). Then it holds that
\begin{equation}\label{supplement1}
||\varepsilon^{-1/2}\sigma_{\mathcal{T}}+\varepsilon^{1/2}\nabla u_{h}|| \preceq |||u-u_{h}|||_{*}+{\rm osc}_{h}.
\end{equation}
\end{lemma}
\begin{proof}
For all $\tau_{\nu}\in\mathcal{V}$, we have from (\ref{recover21}) that
\begin{align}
    \nonumber
    ||\varepsilon^{-1/2}\sigma_{\mathcal{T}} + & \varepsilon^{1/2}\nabla u_{h}||^{2} =
    (\varepsilon^{-1/2} \sigma_{\mathcal{T}} + \varepsilon^{1/2} \nabla u_{h}, \varepsilon^{-1/2}\tau_{\nu} +
    \varepsilon^{1/2}\nabla u_{h}) \\
    \nonumber
    & + (\varepsilon^{-1/2}\sigma_{\mathcal{T}} + \varepsilon^{1/2}\nabla u_{h}, \varepsilon^{-1/2} (\sigma_{\mathcal{T}}-\tau_{\nu})) \\
    \nonumber
    =~ & (\varepsilon^{-1/2}\sigma_{\mathcal{T}} + \varepsilon^{1/2}\nabla u_{h}, \varepsilon^{-1/2}\tau_{\nu} +
    \varepsilon^{1/2}\nabla u_{h}) \\
    \label{recover28}
    & + \sum_{K\in\mathcal{T}_{h}} \gamma_{K} (f-{\bf a}\cdot\nabla u_{h}-bu_{h}-
    \nabla\cdot\sigma_{\mathcal{T}}, \nabla\cdot(\sigma_{\mathcal{T}}-\tau_{\nu}))_{K}.
\end{align}
An inverse estimate leads to
\begin{align}
    \nonumber
    & (f-{\bf a}\cdot\nabla u_{h}-bu_{h}-\nabla\cdot\sigma_{\mathcal{T}}, \nabla\cdot(\sigma_{\mathcal{T}}-\tau_{\nu}))_{K} \\
    \nonumber
    = & \ (R_{K}-(\varepsilon\triangle u_{h}+\nabla\cdot\sigma_{\mathcal{T}}),
    \nabla\cdot(\sigma_{\mathcal{T}}-\tau_{\nu}))_{K} \\
    \label{recover29}
    \leq & (||R_{K}||_{K}+Ch_{K}^{-1}\varepsilon^{1/2}||\varepsilon^{1/2}\nabla u_{h}+
    \varepsilon^{-1/2}\sigma_{\mathcal{T}}||_{K}) Ch_{K}^{-1}\varepsilon^{1/2}||\varepsilon^{-1/2}(\sigma_{\mathcal{T}}-\tau_{\nu})||_{K}.
\end{align}

Choose $\gamma_{K}>0$ to satisfy
\begin{equation*}
    \gamma_{K}\leq h_{K} \min \Big\{\frac{1}{||{\bf a}||_{L^{\infty}(K)}}, \frac{1}{\sqrt{\beta\varepsilon}},
    \frac{\alpha_{K}}{8C^{2}\sqrt{\varepsilon}} \Big\} \quad \forall K\in\mathcal{T}_{h}.
\end{equation*}
From (\ref{recover28})-(\ref{recover29}), Young'inequality, $\alpha_{K}\leq h_{K}/\sqrt{\varepsilon}$, and triangle inequality, we have
\begin{align*}
    & ||\varepsilon^{-1/2}\sigma_{\mathcal{T}}+\varepsilon^{1/2}\nabla u_{h}||^{2} \leq
    \frac{1}{8} ||\varepsilon^{-1/2}\sigma_{\mathcal{T}}+\varepsilon^{1/2}\nabla u_{h}||^{2} +
    2||\varepsilon^{-1/2}\tau_{\nu}+\varepsilon^{1/2}\nabla u_{h}||^{2} \\
    & \quad \quad + \sum_{K\in\mathcal{T}_{h}} \Big(\frac{1}{8C}\alpha_{K}||R_{K}||_{K} +
    \frac{1}{8}||\varepsilon^{-1/2}\sigma_{\mathcal{T}}+
    \varepsilon^{1/2}\nabla u_{h}||_{K} \Big) ||\varepsilon^{-1/2}(\sigma_{\mathcal{T}}-\tau_{\nu})||_{K} \\
    \leq & \frac{3}{8}||\varepsilon^{-1/2}\sigma_{\mathcal{T}} + \varepsilon^{1/2}\nabla u_{h}||^{2} +
    \frac{17}{8}||\varepsilon^{1/2}\nabla u_{h}+\varepsilon^{-1/2}\tau_{\nu}||^{2} +
    \frac{1}{8C^{2}} \sum_{K\in\mathcal{T}_{h}} \alpha_{K}^{2} ||R_{K}||_{K}^{2},
\end{align*}
which results in, for all $\tau_{\nu}\in\mathcal{V}$,
\begin{equation*}
    ||\varepsilon^{-1/2}\sigma_{\mathcal{T}}+\varepsilon^{1/2}\nabla u_{h}||^{2}
    \leq \frac{17}{5} ||\varepsilon^{1/2}\nabla u_{h} + \varepsilon^{-1/2}\tau_{\nu}||^{2} +
    \frac{1}{5C^{2}} \sum_{K\in\mathcal{T}_{h}} \alpha_{K}^{2} ||R_{K}||_{K}^{2}.
\end{equation*}
Therefore,
\begin{equation*}
    ||\varepsilon^{-1/2}\sigma_{\mathcal{T}} + \varepsilon^{1/2}\nabla u_{h}||^{2} \leq
    \frac{17}{5} \min_{\tau_{\nu}\in\mathcal{V}} ||\varepsilon^{1/2}\nabla u_{h} +
    \varepsilon^{-1/2}\tau_{\nu}||^{2} + \frac{1}{5C^{2}} \sum_{K\in\mathcal{T}_{h}} \alpha_{K}^{2} ||R_{K}||_{K}^{2}.
\end{equation*}
By taking $\tau_{\nu}$ obtained by the implicit approximation (\ref{recover2}) or the
explicit approximation (\ref{recover4}), and using the fact that $RT_{0}\subset BDM_{1}$ and
Lemmas~\ref{recover15}-\ref{yingli5.2}, we obtain the assertion (\ref{supplement1}).
\end{proof}

\begin{theorem}
Let $u$ and $u_{h}$ be the solutions to (\ref{PDE3}) and (\ref{PDE6}), respectively, and $\sigma_{\mathcal{T}}$ be the
$H(\rm div)$ recovery flux obtained from (\ref{recover21}). Then there holds
\begin{align}
    \label{recover27}
    & \Big( \sum_{K\in\mathcal{T}_{h}}\alpha_{K}^{2}(||R_{K}||_{K}^{2} + ||\tilde{R}_{K}||_{K}^{2}) +
    ||\varepsilon^{1/2}\nabla u_{h}+\varepsilon^{-1/2}\sigma_{\mathcal{T}}||^{2}\Big)^{1/2} \\
    \nonumber
    & + \Big( \sum_{e\subset\Gamma_{N}}\alpha_{e}^{2} (||g-\varepsilon\nabla u_{h}\cdot{\bf n}||_{e}^{2} +
    ||(\sigma_{\mathcal{T}}+\varepsilon\nabla u_{h})\cdot{\bf n}||_{e}^{2} ) \Big)^{1/2}
    \preceq |||u-u_{h}|||_{*}+{\rm osc}_{h}.
\end{align}
\end{theorem}
\begin{proof}
Collecting Lemma~\ref{recover15}, and Lemmas~\ref{aaa8}-\ref{supplement}, we obtain the estimate (\ref{recover27}).
\end{proof}

\begin{remark}[On three recovering approaches]
First, note that the explicit recovering does not require solving an algebraic system, which, however, is
demanded by the implicit and $H({\rm div})$ approaches. From the perspective of accuracy, the implicit
and $H({\rm div})$ recoveries are intuitively better than the explicit scheme. $L^{2}$-projection
recovery is a special case of $H({\rm div})$ recovering in which the stabilization
parameter $\gamma_{K}=0$. $H({\rm div})$ recovering is based on mixed FEM, which has been
used to precisely approximate the flux \cite{cai2}. In particularly, for $L^{2}$-projection recovery,
when $\mathcal{V}=BDM_{1}$, following the idea of multipoint flux mixed FEM
in \cite{Yotov06}, one concludes that the cost of solving an algebraic system is equivalent to that for computing the estimator.
\end{remark}

\section{Numerical experiments}\label{NumExp}
In this section, we will demonstrate the performance of our {\em a posteriori} error estimators in two example problems.

\subsection{Example 1: boundary layer}
In this example, we take $\Omega=(0,1)^2$, ${\bf a}=(1,1)$, and $b=1$. We use $\beta=1$ and set the right-hand side $f$ so that the exact solution of (\ref{PDE1}) is
\begin{equation*}
    u(x,y) = \Big( \frac{\exp(\frac{x-1}{\varepsilon})-1}{\exp(-\frac{1}{\varepsilon})-1}+x-1 \Big)
    \Big( \frac{\exp(\frac{y-1}{\varepsilon})-1}{\exp(-\frac{1}{\varepsilon})-1}+y-1 \Big).
\end{equation*}
Clearly, $u$ is $0$ on $\Gamma$ and has boundary layers of width $\mathcal{O}(\varepsilon)$ along $x=1$ and $y=1$. Note that for a fixed $\varepsilon$, similar as in \cite{AUGUSTIN}, one can numerically compute the characteristic layers. However, we shall be focused on numerical robustness of the estimators in this paper.

The coarsest triangulation $\mathcal{T}_{0}$ is obtained from halving 4 congruent squares by connecting the bottom right and top left corners. We employ D\"{o}rfler strategy with the marking parameter $\theta=0.5$, and use the ``longest edge" refinement to obtain an admissible mesh.

\begin{figure}[t]
    \centering
    \includegraphics[width=5in]{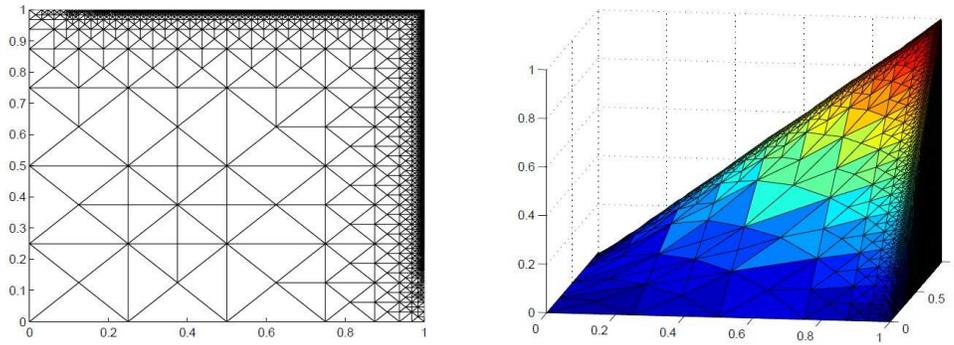}
    \caption{An adaptive mesh with 54855 triangles (left) and the approximation of displacement (piecewise linear element) on the corresponding mesh (right) for $\varepsilon=10^{-12}$ by using the estimator from \eqref{recover4}.}
    \label{Fig.Eg1.2}
\end{figure}

\begin{figure}[t]
    \centering
    \includegraphics[width=5in]{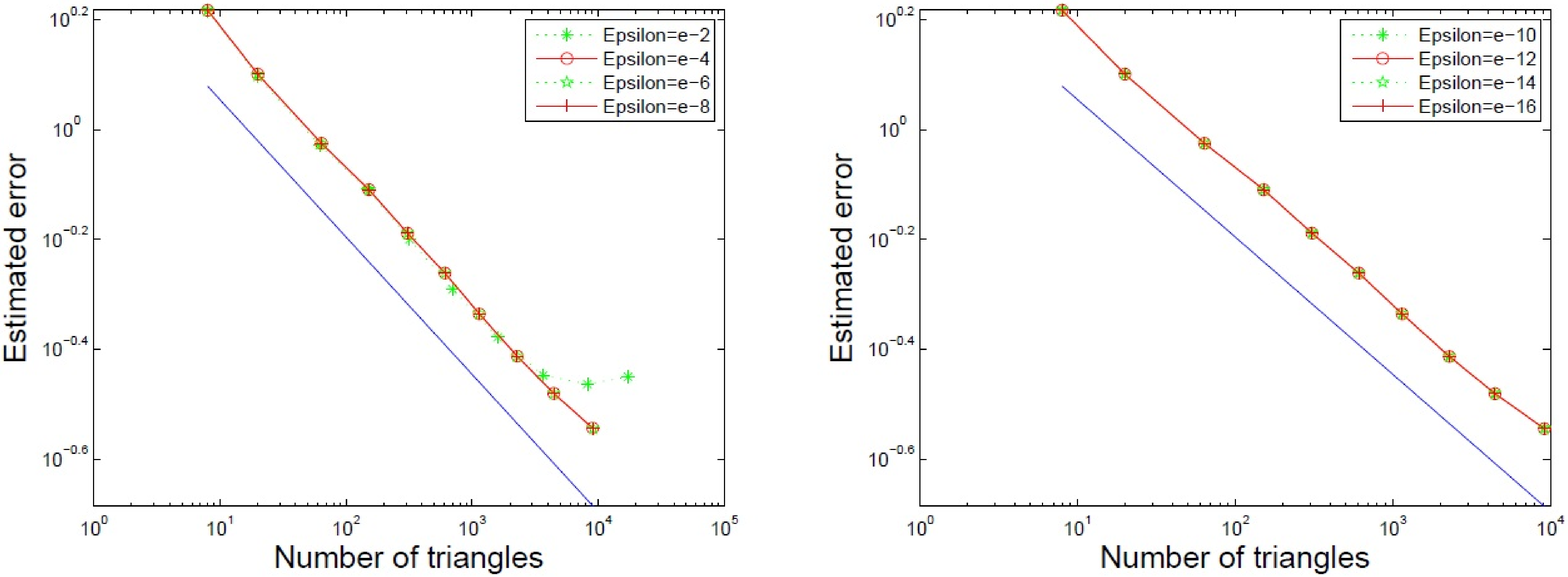}
    \caption{Estimated error of the flux against the number of elements in adaptively refined
    meshes for $\varepsilon$ from $10^{-2}$ to $10^{-8}$ (left) and from $10^{-10}$ to $10^{-16}$ (right) by
    using the estimator from \eqref{recover4}.}
    \label{Fig.Eg1.3}
\end{figure}

\begin{figure}[t]
    \centering
    \begin{minipage}[t]{0.92\linewidth}
        \includegraphics[width=0.92\linewidth]{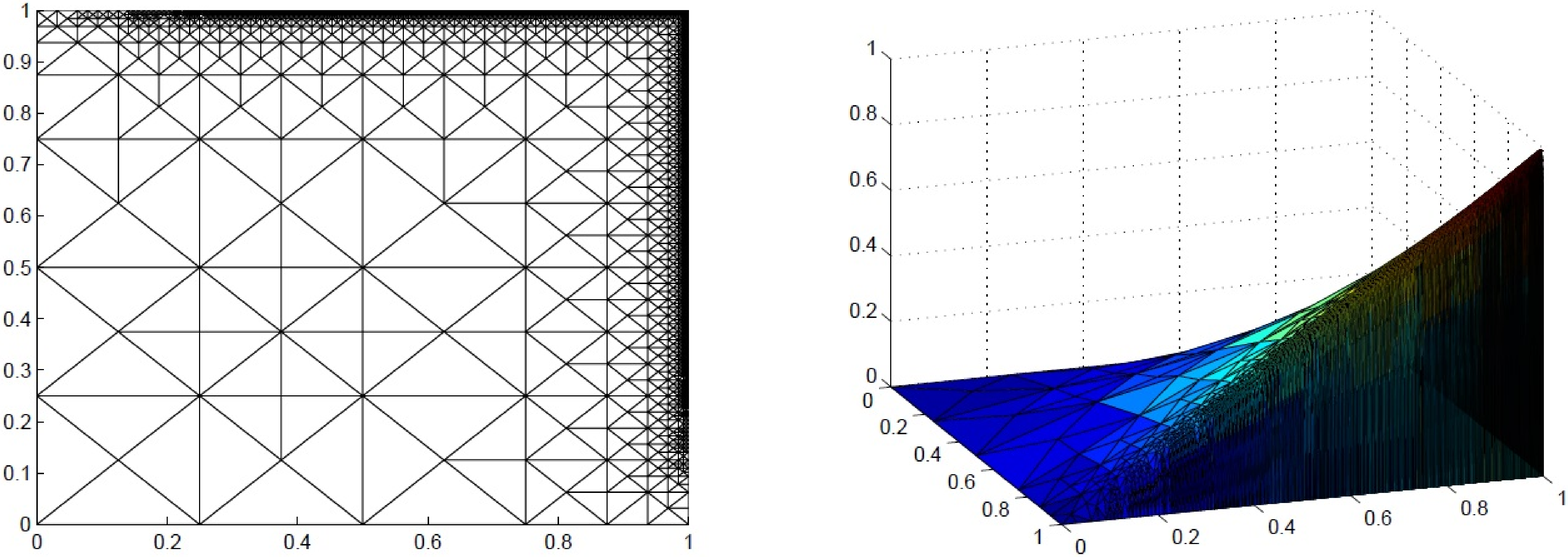}
    \end{minipage}
    \caption{An adaptive mesh with 30869 triangles (left) and the approximation of displacement (piecewise linear element) on the corresponding mesh (right) for $\varepsilon=10^{-16}$ by using the estimator from \eqref{recover2}.}
    \label{Fig.Eg1.4}

    \begin{minipage}[t]{0.92\linewidth}
        \includegraphics[width=0.92\linewidth]{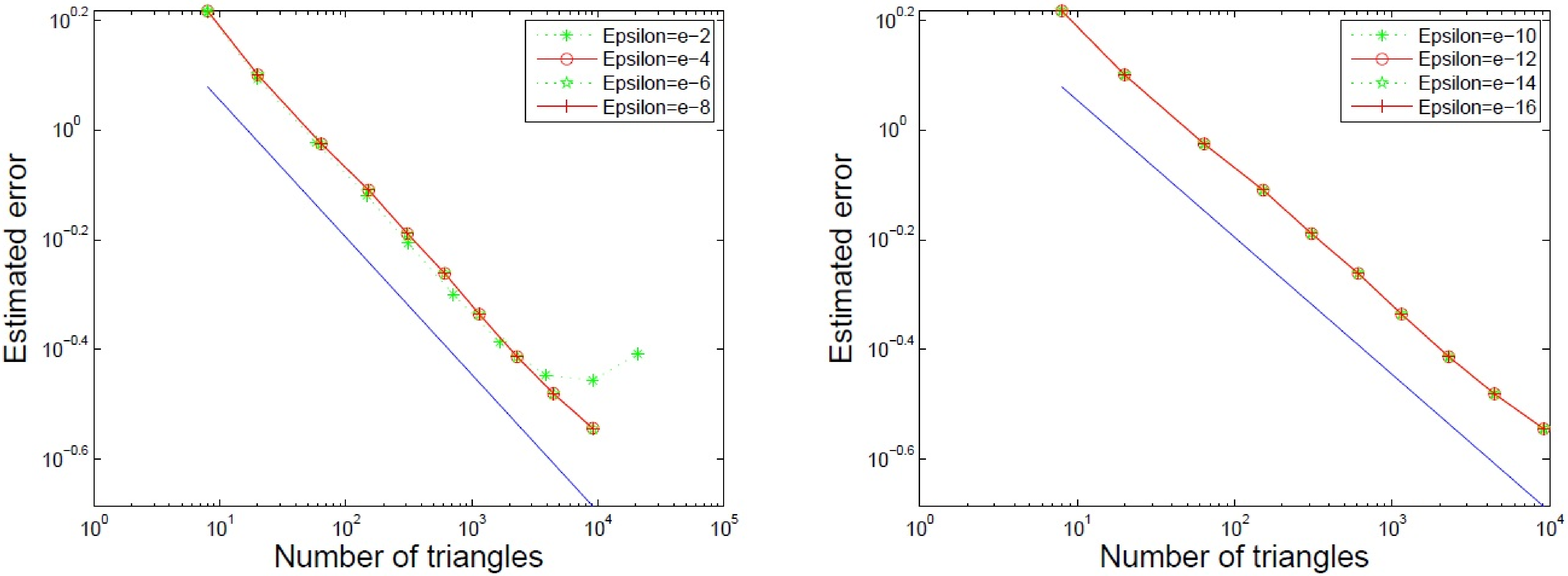}
    \end{minipage}
    \caption{Estimated error of the flux against the number of elements in adaptively refined meshes for $\varepsilon$ from $10^{-2}$ to $10^{-8}$ (left) and from $10^{-10}$ to $10^{-16}$ (right) by using the estimator from \eqref{recover2}.}
    \label{Fig.Eg1.5}
\end{figure}

\begin{figure}[t]
    \centering
    \begin{minipage}[t]{0.92\linewidth}
        \includegraphics[width=0.92\linewidth]{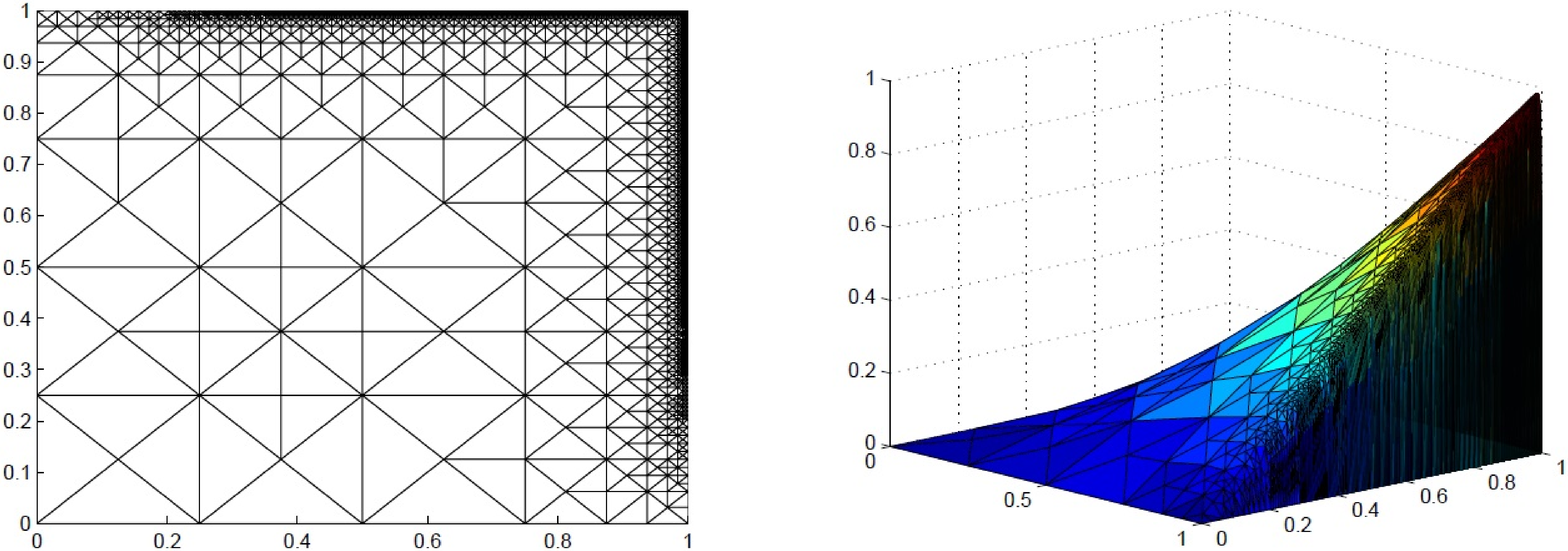}
    \end{minipage}
    \caption{An adaptive mesh with 17382 triangles (left) and the approximation of displacement (piecewise linear element) on the corresponding mesh (right) for $\varepsilon=10^{-16}$ by using the estimator from (\ref{recover21}).}
    \label{Fig.Eg1.6}

    \begin{minipage}[t]{0.92\linewidth}
        \includegraphics[width=0.92\linewidth]{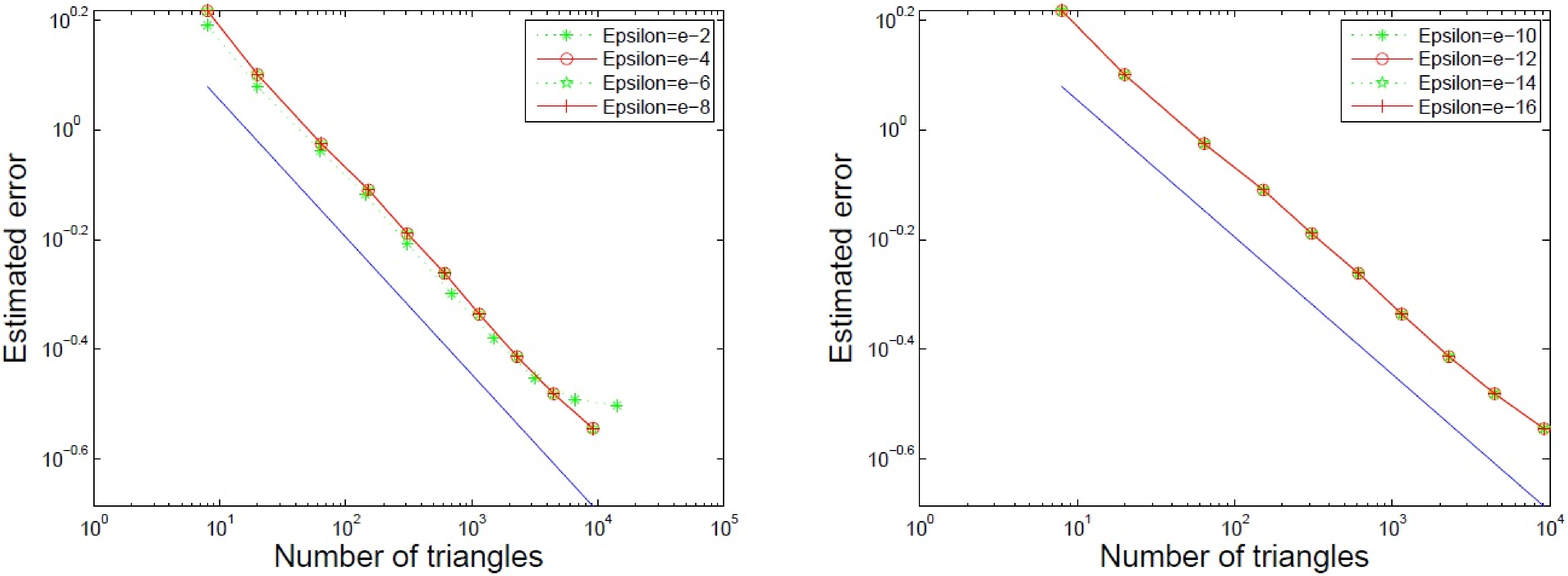}
    \end{minipage}
    \caption{Estimated error of the flux against the number of elements in adaptively refined meshes
    for $\varepsilon$ from $10^{-2}$ to $10^{-8}$ (left) and from $10^{-10}$ to $10^{-16}$ (right) by
    using the estimator from (\ref{recover21}).}
    \label{Fig.Eg1.7}

    \begin{minipage}[t]{0.92\linewidth}
        \includegraphics[width=0.92\linewidth]{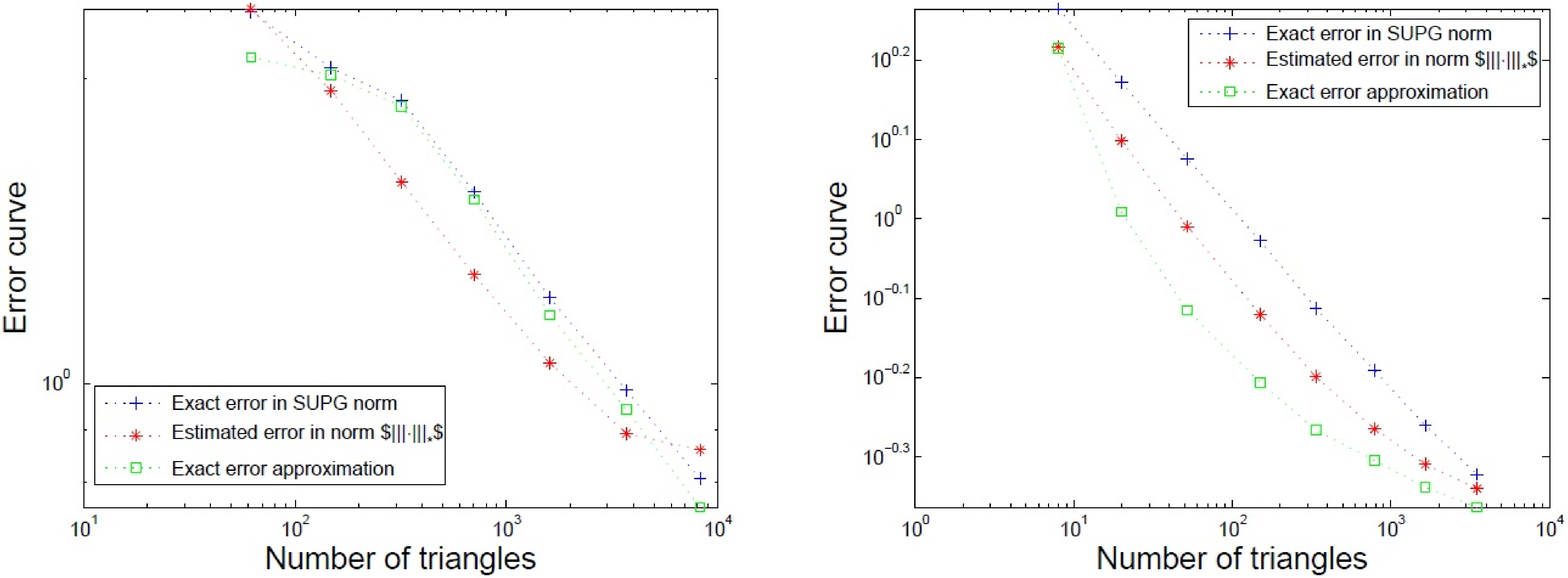}
    \end{minipage}
    \caption{\it Exact error in SUPG norm, estimated error in norm $|||\cdot|||_{*}$, and exact error
    approximation in norm $|||\cdot|||_{\varepsilon}$ for explicit recovering
    for Example 1 with $\theta=0.5$, $\varepsilon=10^{-2}$, and $\delta_{K}=16h_{K}$ (left),
    and $\theta=0.5$, $\varepsilon=10^{-6}$, and $\delta_{K}=4h_{K}$ (right).}
     \label{aaaabb}
\end{figure}

In Figures~\ref{Fig.Eg1.2}, \ref{Fig.Eg1.4}, and \ref{Fig.Eg1.6}, we plot adaptive meshes and numerical
displacements by using the estimators obtained from the explicit recovery \eqref{recover4}, the $L^{2}$-projection
recovery \eqref{recover2}, and the $H({\rm div})$ recovery (\ref{recover21}), respectively. Here the stabilization parameter
is chosen as $\delta_{K}=h_{K}$ on each element $K\in\mathcal{T}_{h}$. Note that the constant $C$ in
the stabilization parameter $\gamma_{K}$ in $H({\rm div})$ recovery (\ref{recover21}) is taken as $C=1$ throughout numerical experiments.

It is observed that strong mesh refinements occur along $x=1$ and $y=1$, where the estimators correctly capture boundary layers and resolve them in convection-dominated regimes. Figures~\ref{Fig.Eg1.3}, \ref{Fig.Eg1.5}, and \ref{Fig.Eg1.7}, which are respectively in correspondence to (4.1), (4.2), and \eqref{aaa9}, report the estimated error against the number of elements in adaptively refined meshes obtained by using estimators from flux recoveries \eqref{recover4}, \eqref{recover2}, and (\ref{recover21}), respectively. Here $\delta_{K}=16 h_{K}$, the errors are measured in $|||\cdot|||_{*}$, and $\varepsilon$ is from $10^{-2}$ to $10^{-16}$. It is observed that the estimated errors depend on $DOF$ uniformly in $\varepsilon$. The estimators work well even if P\'{e}clet number is large, and the estimated errors of all three cases are convergent. As indicated in Remark~\ref{Rk:Norms}, we substitute $|||u-u_{h}|||_{*}$ with $||u-u_{h}||_{\rm SUPG}$ or $|||u-u_{h}|||_{\varepsilon}$ to compute the effectivity indices. We point out that the performance of the true error $|||u-u_{h}|||_{*}$ is between that of $|||u-u_{h}|||_{\varepsilon}$ and $||u-u_{h}||_{{\rm SUPG}}$ up to a multiple of a constant independent of $h$ and $\varepsilon$. To confirm this assertion, Figure~\ref{aaaabb} illustrates $||u-u_{h}||_{{\rm SUPG}}$, the estimated error, and $|||u-u_{h}|||_{\varepsilon}$. It is observed that, in the convection-dominated regime, the behavior of the true error is very similar to that of $|||u-u_{h}|||_{\varepsilon}$ and $||u-u_{h}||_{{\rm SUPG}}$. Thus, it is reasonable to use $|||u-u_{h}|||_{\varepsilon}$ or $||u-u_{h}||_{{\rm SUPG}}$ to approximate the true error $|||u-u_{h}|||_{*}$ when convection dominates. In Table~\ref{Tab:EgOne1}, we show numerical results for implicit $L^{2}$-projection recovering for $\varepsilon=10^{-6}$, $\theta=0.5$, and $\delta_{K}=4h_{K}$. The effectivity indices (ratio of estimated and exact errors) are close to $1$ after 8 iterations. Moreover, the estimators are robust with respect to $\varepsilon$.

We have checked the cases for $\delta_{K}$ from $\delta_{K}=h_{K}$ to $\delta_{K}=16h_{K}$, and found that the choice of $\delta_{K}$ has a slight influence to the quality of the mesh. This observation indicates that adaptivity and stabilization for convection-diffusion equation is worthy of further study. In fact, the current state-of-the-art in stabilization is not completely satisfactory. In particular, the choice of stabilization parameters is still a subtle issue that is not fully understood. This is reflected either by remaining unphysical oscillations in the numerical solution or by smearing solution features too much. For more discussion on
this subject, we refer to \cite{Cohen}.

\begin{table}[t]\small
    \begin{center}
        \caption{Example 1: $k$ -- the number of iterations; $\eta_{k}$ -- the estimated numerical error in $|||\cdot|||_{*}$; $err_{\rm SUPG}$ and eff-index$_1$ -- the exact error in $||\cdot||_{\rm SUPG}$ and the corresponding effectivity index; and $err_{\rm APP}$ and eff-index$_{2}$ -- the exact approximation error in $|||\cdot|||_{\varepsilon}$ and the corresponding effectivity index. Here $\varepsilon=10^{-6}$, $\theta=0.5$, and $\delta_{K}=4h_{K}$.}
        \label{Tab:EgOne1}
        \small 
        \begin{tabular}{|c|c|c|c|c|c|c|c|c|} \hline
            $k$& $1$& $2$& $3$& $4$& $5$& $6$& $7$& $8$\\ \hline
            $\eta_{k}$&1.6476&1.2551&0.9790&0.7580&0.6328&0.5438&0.4908&0.4573\\ \hline
            $err_{\rm SUPG}$&1.8419&1.4871&1.1914&0.9387&0.7717&0.6445&0.5492&0.4765\\ \hline
            eff-index$_{1}$&0.8945&0.8440&0.8217&0.8075&0.8200&0.8437&0.8936&0.9598\\ \hline
            $err_{\rm APP}$&0.8204&0.6811&0.5760&0.4971&0.4519&0.4253&0.4018&0.3852\\ \hline
            eff-index$_{2}$&2.0083&1.8427&1.6998&1.5230&1.4002&1.2786&1.2216&1.1872\\ \hline
        \end{tabular}
        \bigskip
        \caption{Example 2: Numerical results by \eqref{recover4} with $\delta_{K}=16h_{K}$ and $\theta=0.3$. In the table, $\varepsilon$ is the singular perturbation parameter, $\eta_{k}$ is the estimated numerical error, TOL is the given tolerance, DOF is the degrees of freedom, $h_{\rm max}(\varepsilon)$ and $h_{\rm min}(\varepsilon)$ are respectively the largest and smallest mesh sizes, and $k$ is the number of iterations.}
        \label{Tab:EgTwo2}
        \small
        \begin{tabular}{|l|c|c|c|c|c|}
            \hline
            $\varepsilon$ & $10^{-6}$ & $10^{-5}$ & $10^{-4}$ & $10^{-3}$ & $10^{-2}$ \\ \hline
            $\eta_{k}$/TOL & 24.159 & 29.5353 & 51.0934 & 86.1609 & 154.3741 \\ \hline
            DOF & 28194 & 15696 & 1999 & 368 & 75 \\ \hline
            $h_{\rm max}(\varepsilon)$ & 0.55902 & 0.55902 & 0.55902 & 0.55902 & 1.118034 \\ \hline
            $h_{\rm min}(\varepsilon)$ & 7.63e-06 & 3.05e-05 & 4.88e-04 & 3.91e-03 & 0.031250 \\ \hline
            $k$ &  20 & 18 & 14 & 11 & 8 \\ \hline
        \end{tabular}
    \end{center}
\end{table}

\subsection{Example 2: interior and boundary layer}

This model problem is one of the examples solved by Verf\"{u}rth in
ALF software. Let $\Omega=(-1,1)^2$. We set the velocity
field ${\bf a}=(2,1)$, the reaction coefficient $b=0$, and the
source term $f=0$ in (\ref{PDE1}), and consider cases for
$\varepsilon$ from $10^{-3}$ to $10^{-15}$. The following Dirichlet
boundary conditions are applied: $u(x,y)=0$ along $x=-1$ and $y=1$,
and $u(x,y)=100$ along $x=1$ and $y=-1$. The exact solution of this
problem is not available, which however exhibits an exponential
boundary layer along the boundary $\{ (x,y) \; : \: x=1, y>0 \}$,
and a parabolic interior layer along the line segment connecting points
$(-1,-1)$ and $(1,0)$. Note that the interior layer extends in the
direction of the convection coefficient.

\begin{figure}[t]
    \centering
    \begin{minipage}[t]{0.92\linewidth}
        \includegraphics[width=0.92\linewidth]{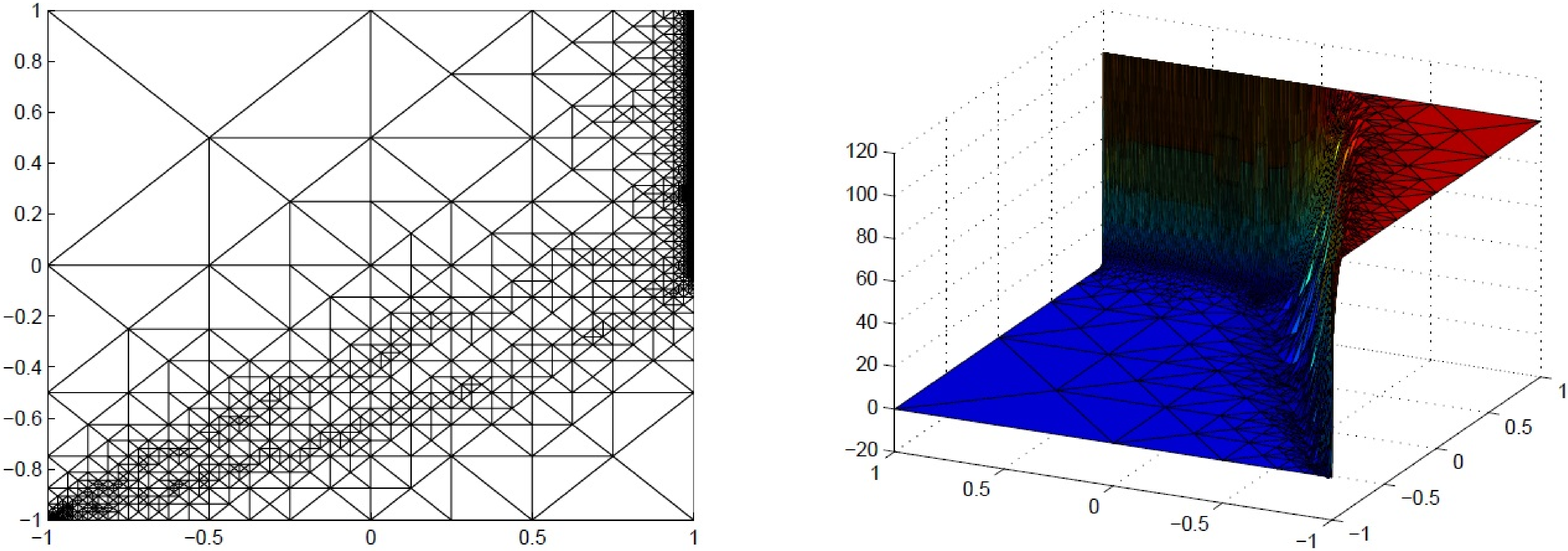}
    \end{minipage}
    \caption{An adaptive mesh with 14315 triangles (left) and the approximation of displacement (piecewise linear element) on the corresponding mesh (right) for $\varepsilon=10^{-11}$ by using the estimator from \eqref{recover4}.}
    \label{Fig.Eg2.1}

    \begin{minipage}[t]{0.92\linewidth}
        \includegraphics[width=0.92\linewidth]{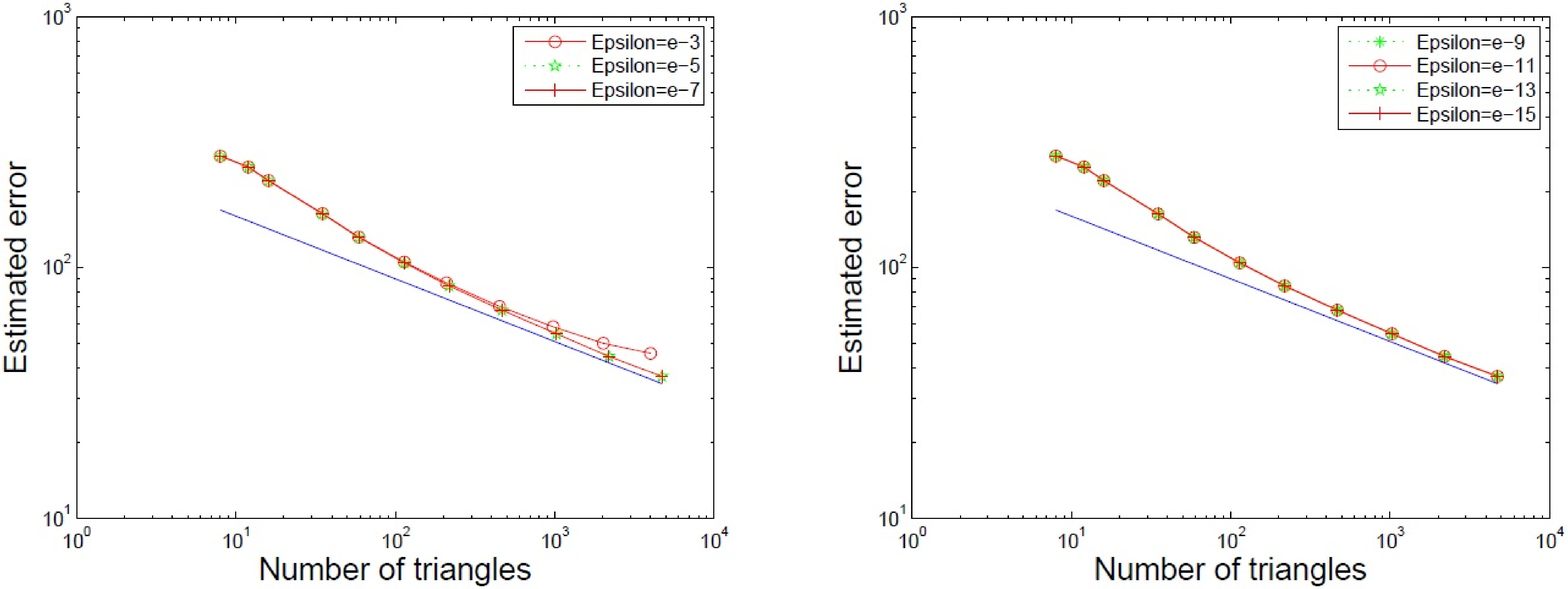}
    \end{minipage}
    \caption{Estimated error of the flux against the number of elements in adaptively refined
    meshes for $\varepsilon$ from $10^{-3}$ to $10^{-7}$ (left) and from $10^{-9}$ to $10^{-15}$ (right) by
    using the estimator from \eqref{recover4}.}
    \label{Fig.Eg2.3}
\end{figure}

\begin{figure}[t]
    \centering
    \begin{minipage}[t]{0.92\linewidth}
        \includegraphics[width=0.92\linewidth]{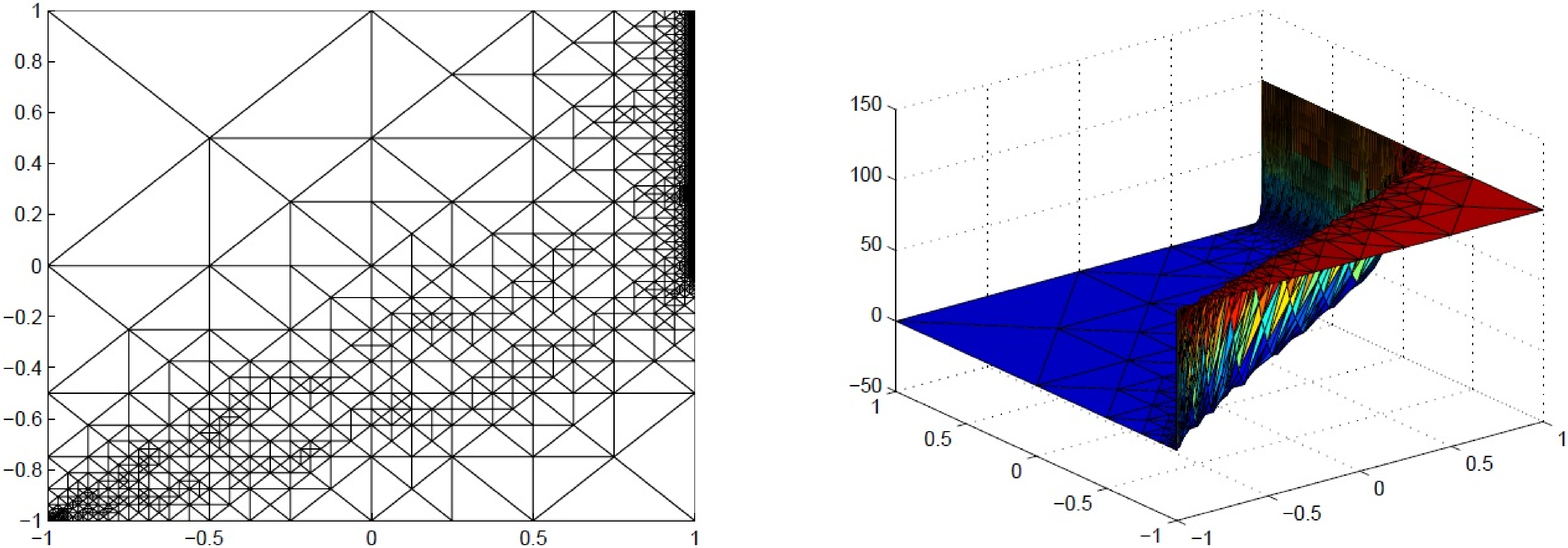}
    \end{minipage}
    \caption{An adaptive mesh with 7761 triangles (left) and the approximation of displacement (piecewise linear element) on the corresponding mesh (right) for $\varepsilon=10^{-11}$ by using the estimator from \eqref{recover2}.}
    \label{Fig.Eg2.4}

    \begin{minipage}[t]{0.92\linewidth}
        \includegraphics[width=0.92\linewidth]{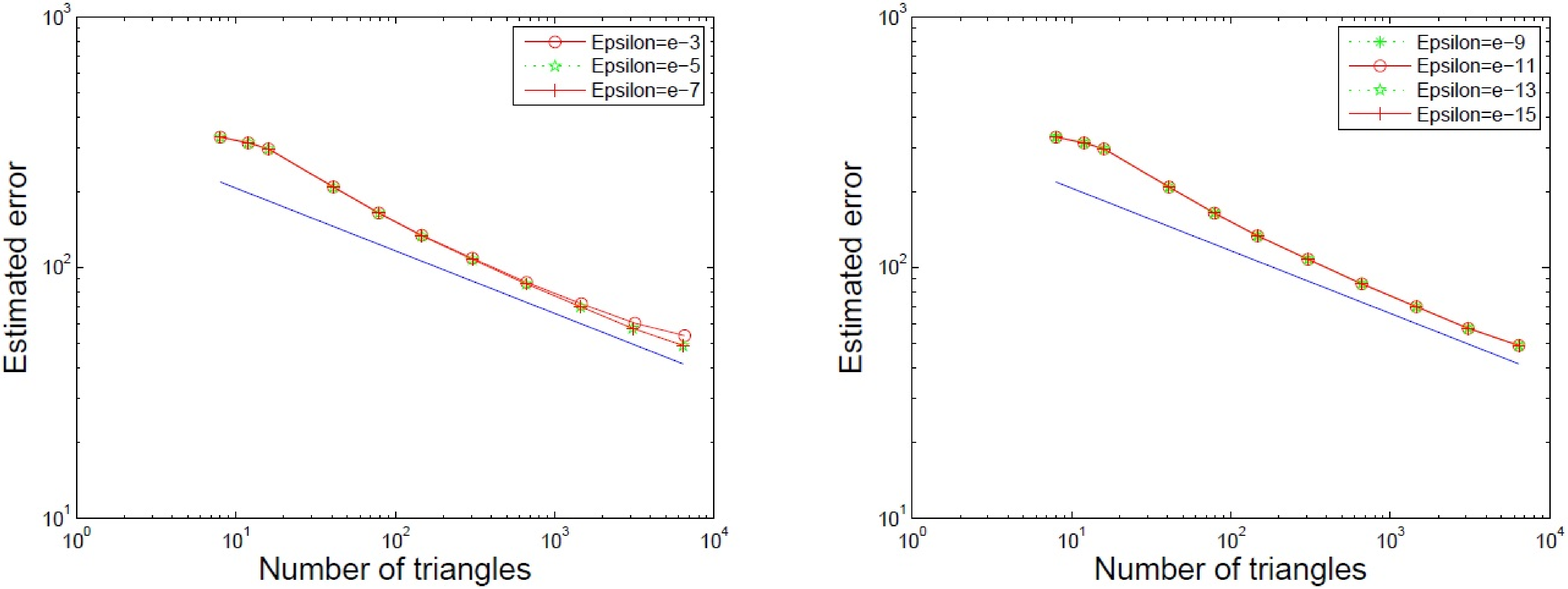}
    \end{minipage}
    \caption{Estimated error of the flux against the number of elements in adaptively refined
    meshes for $\varepsilon$ from $10^{-3}$ to $10^{-7}$ (left) and from $10^{-9}$ to $10^{-15}$ (right) by
    using the estimator from \eqref{recover2}.}
    \label{Fig.Eg2.5}
\end{figure}

\begin{figure}[htbp]
    \centering
    \begin{minipage}[t]{0.92\linewidth}
        \includegraphics[width=0.92\linewidth]{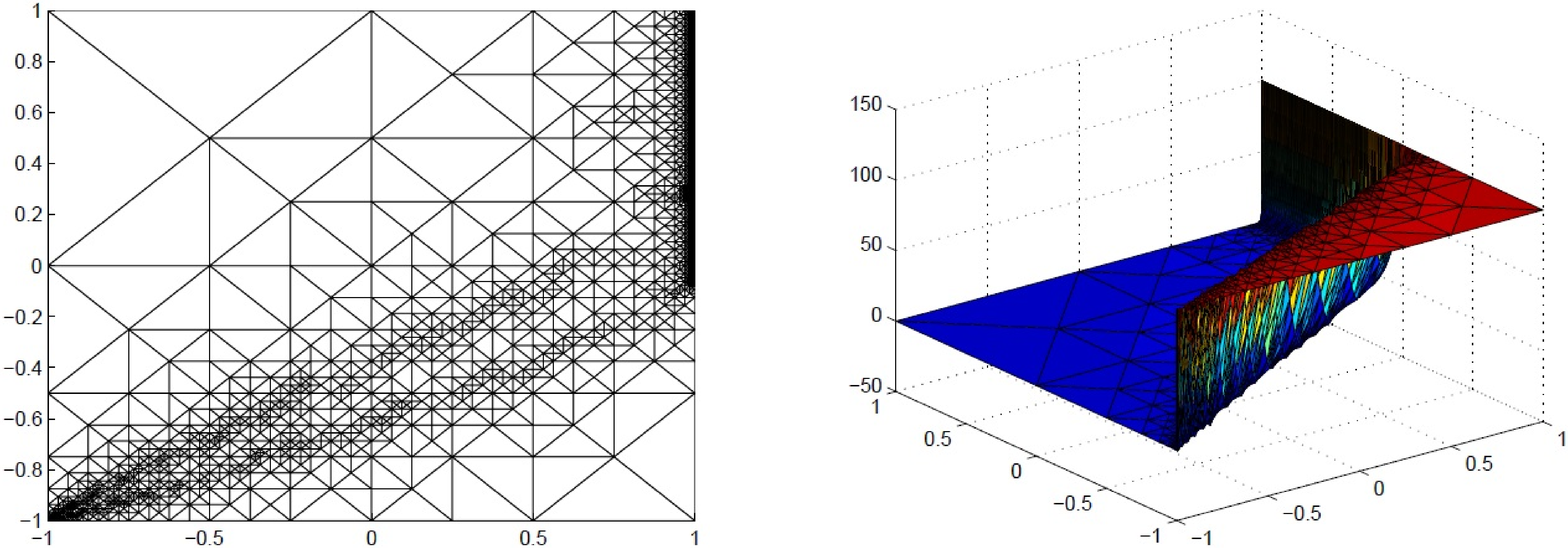}
    \end{minipage}
    \caption{An adaptive mesh with 27309 triangles (left) and the approximation of displacement (piecewise linear element) on the corresponding mesh (right) for $\varepsilon=10^{-11}$ by using the estimator from (\ref{recover21}).}
    \label{Fig.Eg2.6}

    \begin{minipage}[t]{0.92\linewidth}
        \includegraphics[width=0.92\linewidth]{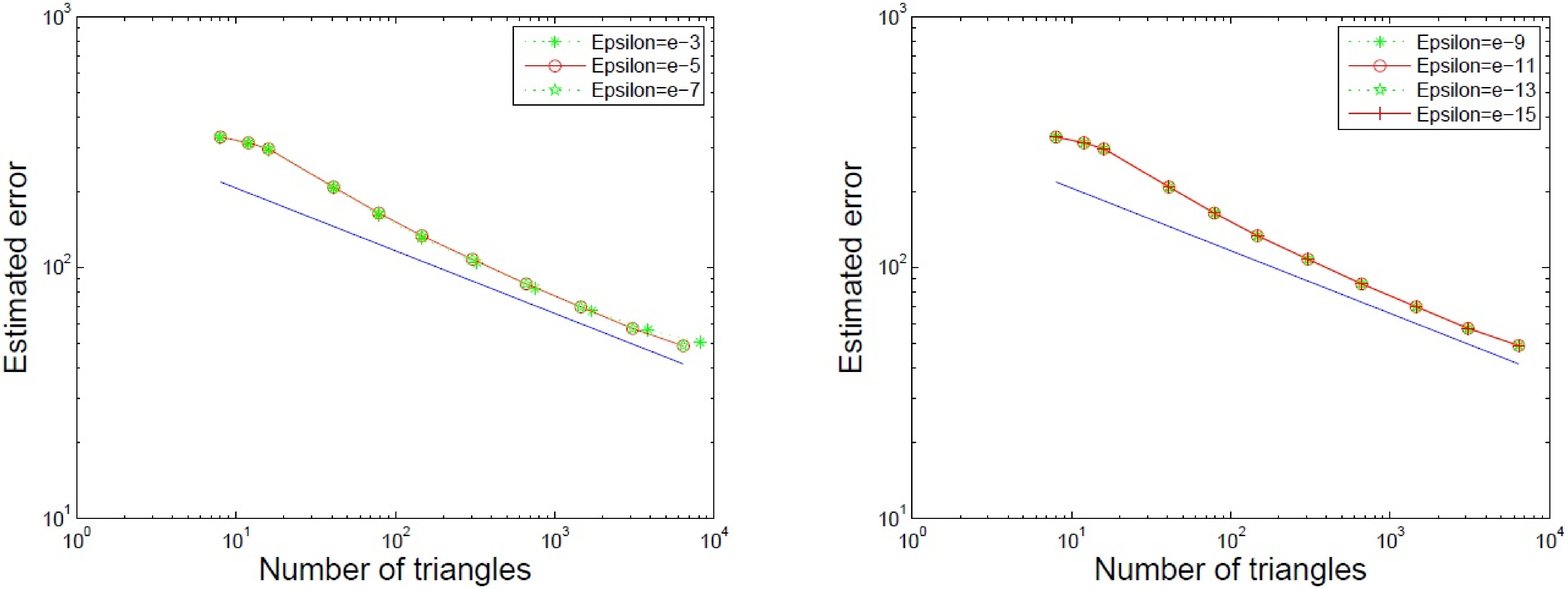}
    \end{minipage}
    \caption{Estimated error of the flux against the number of elements in adaptively refined meshes
    for $\varepsilon$ from $10^{-3}$ to $10^{-7}$ (left) and from $10^{-9}$ to $10^{-15}$ (right) by
    using the estimator from (\ref{recover21}).}
    \label{Fig.Eg2.7}
\end{figure}


We choose the same initial mesh as in Example 1. From Figures \ref{Fig.Eg2.1}, \ref{Fig.Eg2.4},
and \ref{Fig.Eg2.6}, which are respectively depicted by using the estimators obtained from the
explicit recovery \eqref{recover4}, the $L^{2}$-projection recovery \eqref{recover2}, and the $H({\rm div})$
recovery (\ref{recover21}), and by choosing the stabilization parameters as $\delta_{K}=h_{K}$. It is
observed that the meshes are refined in both the exponential and the parabolic layer regions,
but the refinement first occurs in the region near $\{ (x,y) \; : \: x=1, y>0 \}$. The reason is
that the exponential layer is much stronger than the parabolic layer. It is also observed that each
one of three estimators capture the behavior of the solution pretty well, even when the singular
perturbation parameter $\varepsilon$ is very small.

Figures~\ref{Fig.Eg2.3}, \ref{Fig.Eg2.5}, and \ref{Fig.Eg2.7} are
depicted by using the estimators obtained from the flux recovery
\eqref{recover4}, \eqref{recover2}, and (\ref{recover21}), respectively, and by choosing the
stabilization parameters as $\delta_{K}=16h_{K}$. The estimated
error against the number of elements in adaptively refined mesh for
$\varepsilon$ from $10^{-3}$ to $10^{-15}$ are reported. It is
observed that all three estimated errors from respective estimators
in norm $|||\cdot|||_{*}$ reduce uniformly in sufficiently small
$\varepsilon$ in absence of reaction term. In addition, the same
convergence rates as in Example 1 are obtained. It is also noticed
that the performance of the three estimators are similar.

In Table~\ref{Tab:EgTwo2}, data for different $\varepsilon$s are provided. The adaptive iterations refine elements till the layer is resolved or the TOL is met. One may observe that the performance of the error estimators depends on the TOL; the minimum mesh sizes $h_{min}$ are of order $O(\varepsilon h_{max})$ or $O(\varepsilon)$, since the maximum mesh sizes $h_{max}(\varepsilon)$ and the initial mesh size $h_0$ are of similar sizes; the DOF required for resolving layers will increase when TOL and/or $\varepsilon$ decrease; and the proposed error estimators are robust with respect to $\varepsilon$. On the other hand, due to the current state-of-the-art in stabilization, spurious oscillations may occur on very fine mesh, which will hence affect the quality of mesh refinement of further iterations and the rate of convergence of the method; cf. \cite{Cohen} and the plots for $\varepsilon = 10^{-2}$ in Figures~\ref{Fig.Eg1.3} and \ref{Fig.Eg1.5}.




\begin{thebibliography}{99}

\bibitem{ADAMS}
R.A. Adams, {\it Sobolev space}, Academic Press, New York. 1975.

\bibitem{Ani}
M. Ainsworth, A. Allends, G. R. Barrenechea, and R. Rankin, {\it Fully computable a posteriori error bounds for stabilized FEM approximations of convection-reaction-diffusion problems in three dimentions}, Int. J. Numer. Meth. Fluids, 73 (2013), no. 9, 765--790.

\bibitem{Ainsworth}
M. Ainsworth and A.W. Craig, {\it {\em A posteriori} error estimation in finite element method}, Numer. Math., 36 (1991), 429--463.

\bibitem{Ainsworth0}
M. Ainsworth and J.T. Oden, {\it {\em A posteriori} error estimation in finite element analysis}, Pure Appl. Math., Wiley-Interscience, John Wiley $\&$ Sons, New York, 2000.

\bibitem{AUGUSTIN}
M. Augustin, A. Caiazzo, A. Fiebach, J. Fuhrmann, V. Jhon, A. Linke, and R. Umla, {\it An assessment of discretizations for convection-dominated convection-diffusion equations}, Comput. Methods Appl. Mech. Engrg., 200 (2011), 3395-3409.

\bibitem{Bank}
R. Bank and J. Xu, {\it Asymptotically exact {\em a posteriori} error estimators, Part 1: Grids with superconvergence}, SIAM J. Numer. Anal., 41 (2003), 2294--2312; {\it Part 2: General unstructured grids}, SIAM J. Numer. Anal., 41 (2003), 2313--2332.

\bibitem{Bank2}
R. Bank, J. Xu, and B. Zheng, {\it Superconvergent derivative recovery for Lagrange triangular elmements of degree $p$ on unstructured grids}, SIAM J. Numer. Anal., 45 (2007), 2032--2046.

\bibitem{Berrone02}
S. Berron, {\it Robustness in {\em a posteriori} error analysis for FEM flow models}, Numer. Math. 91 (2002), no. 3, 389--422.

\bibitem{cai1}
Z. Cai and S. Zhang, {\it Recovery-based error estimator for interface problems: conforming linear elemnts}, SIAM J. Numer. Anal., 47 (2009), no. 3, 2132--2156.

\bibitem{cai2}
Z. Cai and S. Zhang, {\it Flux recovery and {\em a posteriori} error estimators: conforming elements for scalar elliptic equations}, SIAM J. Numer. Anal., 48 (2010), no. 2, 578--602.

\bibitem{Carstensen}
C. Carstensen, {\it All first-order averaging technique for {\em a posteriori} finite element error control on unstructure grids are efficient and reliable}, Math. Comp., 73 (2003), 1153--1165.

\bibitem{Ciarlet}
P.G. Ciarlet, {\it The finite element method for elliptic problems}, Nort-Holland, Amsterdam, 1978.

\bibitem{Clement}
P. Clem\'{e}nt, {\it Approximation by finite element functions using local regularization}, RAIRO S\'{e}r. Rouge Anal. Num\'{e}r., 2 (1975), 77--84.

\bibitem{Cohen}
A. Cohen, W. Dahmen, and G. Welper, {\it Adaptivity and variational stablization for convection-diffusion equations}, ESAIM: Mathematical Modelling and Numerical Analysis, 46 (2012), no. 5, 1247--1273

\bibitem{Du}
S.H. Du and Z. Zhang, {\it A robust residual-type {\em a posteriori} error estimator for convection-diffusion equations}, Journal of Scientific Computing, 65 (2015), pp. 138-170.

\bibitem{Franca}
L.P. Franca, S.L. Frey, and T.J.R. Hughes, {\it Stabilized finite element methods I: Application to the advective-diffusive model}, Comput. Methods Appl. Mech. Engrg., 95 (1992), 253--276.

\bibitem{Hughes}
T.J.R. Hughes and A. Brooks, {\it Streamline upwind/Petrov Garlerkin formulations for the convection dominated flows with particular emphasis on the incompressible Navier-Stokes equations}, Comput. Methods Appl. Mech. Engrg., 54 (1982), 199--259.

\bibitem{John-Novo}
V. John and J. Novo, {\it A robust SUPG norm {\em a posteriori} error estimator for stationary convection-diffusion equations}, Comput. Meth. Appl. Mech. Engrg., 255 (2013), 289--305.

\bibitem{Kunert}
G. Kunert, {\it {\em A posteriori} error estimation for convection dominated problems on anisotropic meshes}, Math. Methods Appl. Sci. 26 (2003), no. 7, 589--617.

\bibitem{Ovall1}
J.S. Ovall, {\it Fixing a ``bug'' in recovery-type {\em a posteriori} error estimators}, Technical report 25, Max-Planck-Institute fur Mathematick in den Naturwissenschaften, Bonn, Germany, 2006.

\bibitem{Ovall2}
J.S. Ovall, {\it Two dangers to avoid when using gradient recovery methods for finite element error estimation and adaptivity}, Technical report 6, Max-Planck-Institute fur Mathematick in den Naturwissenschaften, Bonn, Germany, 2006.

\bibitem{Rapin}
G. Rapin and G. Lube, {\it A stabilized scheme for the Lagrange multiplier method for advection-diffusion equations}, Math. Models Methods Appl. Sci. 14 (2004), no. 7, 1035--1060.

\bibitem{Roos}
H.G. Roos, M. Stynes, and L. Tobiska, {\it Robust numerical methods for singularly perturbed differential equations}, Springer-Verlag Berlin Heidelberg, 2008.

\bibitem{San01}
G. Sangalli, {\it A robust {\em a posteriori} estimator for the residual free bubbles method applied to advection-dominated problems}, Numer. Math., 89 (2001), 379--399.

\bibitem{San08}
G. Sangalli, {\it Robust {\em a-posteriori} estimator for advection-diffusion-reaction problems}, Math. Comp., 77 (2008), no. 261, 41--70.

\bibitem{Tobiska}
L. Tobiska and R. Verf\"{u}rth, {\it Robust {\em a-posteriori} error estimates for stablized finite element methods}, IMA J Numer. Anal., (2015), doi: 10.1093$/$imanum$/$dru060.

\bibitem{Verfurth2}
R. Verf\"{u}rth, {\it {\em A posteriori} error estimators for convection-diffusion equations}, Numer. Math., 80 (1998), 641--663.

\bibitem{Verfurth1}
R. Verf\"{u}rth, {\it Robust {\em a posteriori} error estimates for stationary convection-diffusion equations}, SIAM J.Numer. Anal., 43 (2005), 1766--1782.

\bibitem{Yotov06}
M. F. Wheeler and I. Yotov, {\it A multipoint flux mixed finite element method}, SIAM J. Numer. Anal., 44 (2006), 2082--2106.

\bibitem{Zhang1}
Z. Zhang, {\it {\em A posteriori} error estimates on irregular grids based on gradient recovery}, Adv. Comput. Math., 15 (2001), 363--374.

\bibitem{Zhang}
Z. Zhang, {\it Recovery techniques in finite element methods}, in {\it Adaptive Computations: Theory and Algorithms}, Mathematics Monogr. Ser. 6, T. Tang
and J. Xu, eds., Science Publisher, New York, 333--412, 2007.

\bibitem{Zienkiewicz}
O.C. Zienkiewicz and J.Z. Zhu, {\it A simple error estimator and adaptive procedure for practical engineering analysis}, Internat. J. Numer. Methods Engrg., 24 (1987), 337--357.

\bibitem{Zienkiewicz2}
O.C. Zienkiewicz and J.Z. Zhu, {\it The superconvergent patch recovery and {\em a posteriori} error estimates, Part 1: The recovery technique}, Internat. J. Numer. Methods Engrg., 33 (1992), 1331--1364; {\it Part 2: Error estimates and adaptivity}, Internat. J. Numer. Methods Engrg., 33 (1992), 1365--1382.

\end{thebibliography}
\end{document}